\documentclass{elsarticle}

\usepackage{lineno,hyperref}
\modulolinenumbers[5]

\journal{}

\usepackage{geometry}
\usepackage{amsmath}
\usepackage{lineno,hyperref}
\usepackage{subfigure}
\usepackage{latexsym}
\usepackage{amsmath}
\usepackage{amssymb}
\usepackage{amsfonts}
\usepackage{verbatim}
\usepackage[latin1]{inputenc}
\usepackage{mathrsfs}
\usepackage{amsfonts}
\usepackage{graphicx}
\usepackage{colortbl,dcolumn}
\usepackage{amsmath}
\usepackage{psfrag}
\usepackage{booktabs}
\usepackage{lipsum}
\usepackage{amsfonts}
\usepackage{graphicx}
\usepackage{epstopdf}
\usepackage{algorithmic}
\ifpdf
  \DeclareGraphicsExtensions{.eps,.pdf,.png,.jpg}
\else
  \DeclareGraphicsExtensions{.eps}
\fi
\usepackage[misc]{ifsym}
\usepackage{enumerate}
\usepackage{ulem}
\usepackage{subfigure}
\usepackage{graphicx}
\usepackage{enumerate}
\usepackage{amsthm}
\usepackage{amsmath,bm}
\usepackage{latexsym}
\usepackage{float}
\usepackage{lineno}
\usepackage{latexsym}
\usepackage{amsmath}
\usepackage{amssymb}

\allowdisplaybreaks[4]

\newtheorem{tm}{Theorem}[section]
\newtheorem{rk}{Remark}[section]

\newtheorem{df}{Definition}[section]
\newtheorem{prop}{Proposition}[section]
\newtheorem{lm}{Lemma}[section]
\newtheorem{cor}{Corollary}[section]
\newtheorem{ex}{Example}[section]

\newcommand{\E}{\mathbb E}

\newcommand{\N}{\mathbb N}
\newcommand{\R}{\mathbb R}

\newcommand{\T}{\mathbb T}

\renewcommand{\>}{\rangle}

\allowdisplaybreaks[4]

\begin{document}

\begin{frontmatter}
\title{What is a stochastic Hamiltonian process on finite graph? \\ An optimal transport answer}
\tnotetext[mytitlenote]{}

\author[cas]{Jianbo Cui}
\ead{jcui82@gatech.edu}

\author[cas]{Shu Liu}
\ead{sliu459@math.gatech.edu}

\author[cas]{Haomin Zhou}
\ead{hmzhou@math.gatech.edu}

\address[cas]{School of Mathematics, Georgia Tech, Atlanta, GA 30332, USA}
\begin{abstract}
We present a definition of  stochastic Hamiltonian process on finite graph via its corresponding density dynamics in Wasserstein manifold. We demonstrate the existence of stochastic Hamiltonian process in many classical discrete  problems, such as the optimal transport problem, Schr\"odinger equation and Schr\"odinger bridge problem (SBP). The stationary and periodic properties of Hamiltonian processes are also investigated in the framework of SBP. 
\end{abstract}


\end{frontmatter}


\section{Introduction}
Hamiltonian systems, including both ordinary or partial differential equations (ODEs or PDEs respectively), are ubiquitous in applications. Their mathematical studies have a long and rich history (see e.g., \cite{Rab78,Arn89,MW89}). Traditionally, the ambient space on which to define a Hamiltonian system is continuous, such as Euclidean space $\R^n$ or smooth manifolds like torus $\T^2$. What is a Hamiltonian process if the underlying space becomes discrete, such as a finite graph? This is the question that we would like to explore within the framework of optimal transport (OT) in this study.

Our motivation to consider this question is 3-fold. Curiosity is at the first place. Secondly, the notion of gradient flow on graph has been investigated extensively using OT theory (see e.g. \cite{Mas11, CHLZ12} and references therein). For example, an irreducible and reversible continuous time Markov chain on graph can be viewed as the gradient flow of entropy with respect to the discrete Wasserstein metric \cite{Mas11}. Naturally, we are inspired to ask whether the concept of Hamiltonian process on graph exists or not. To the best of our knowledge, the Hamiltonian mechanics on graph has not been explored yet. Finally and most importantly, recent developments in several practical problems, which can be defined in both continuous and discrete spaces, demonstrate Hamiltonian principles. They are (i) 
the OT problem (see e.g. \cite{Vil09}),
\begin{equation}\label{min-e}\begin{split}
{W_2^2(\rho_0,\rho_1)}&=\inf_{v}\{\int_0^1\E[|\dot X_t|^2]dt\ : \, 
\dot X_t=v(t,X_t), X_0 \thicksim \rho^0, X_1\thicksim \rho^1\},
\end{split}\end{equation}
(ii) the SBP (see e.g. \cite{Sch31}),
\begin{align}\label{Sch-bri-pro}
\inf_{v}\Big\{\int_0^1\frac 12 \mathbb E[|v(t,X_t)|^2]dt: \dot X_t=v(t,X_t)+\sqrt{\hbar} \dot B_t, \;  X_0\sim \rho^0,\; X_1\sim \rho^1\Big\}
\end{align}
 and (iii)
 the Schr\"odinger equation (see e.g. \cite{Nelson19661079,Mad27,CLZ19}),
\begin{align} \label{Sch-pro}
\inf_{v}\Big\{\int_0^T\frac 12 \mathbb E[|\dot X_t|^2]dt: \dot X_t=v(t,X_t)+\sqrt{\hbar} \dot B_t, \;  X_0\sim \rho^0,\; X_1\sim \rho^1\Big\}.
\end{align} The above formulations are presented in Euclidean space where $v\in \mathbb R^d$ can be any smooth vector field, $X_t$ is a stochastic process with prescribed probability densities $\rho^0$ and $\rho^1$ at time $0$ and $1$ respectively, $B_t$ is the standard Brownian motion and $\hbar>0$ is a constant. 

A common property shared by these problems is that their critical points obey the Hamiltonian principle. For instance, the minimizer of OT problem \eqref{min-e} satisfies a Hamiltonian PDE with the Hamiltonian $H(x,v,t)=\frac 12 |v|^2$ (see e.g. \cite{BB00}). The minimizer of SBP \eqref{Sch-bri-pro} is the solution of a Hamiltonian PDE with $H(x,v)=\frac 12|v|^2-\frac 18\hbar \frac {\delta}{\delta \rho}I(\rho)(t,x)$  where the Fisher information $I(\rho)=\int_{\mathbb R^d}|\nabla \log\rho(x)|^2\rho(x)dx$ (see e.g. \cite{Pav03,Leo14}). Needless to say, the critical point of \eqref{Sch-pro} satisfies the the Schr\"odinger equation, which is a well-known Hamiltonian system. The problems stated in \eqref{min-e}, \eqref{Sch-bri-pro} and \eqref{Sch-pro} can be posed, with nominal changes, on a graph, and the density functions of their critical points have been studied on the Wasserstein manifold (see  \cite{gangbo2019geodesics}, \cite{Leo16, chow20discrete}, \cite{CLZ19}) showing that they satisfy Hamiltonian ODEs. Based on those results, we investigate the properties of stochastic process $X(t)$ and provide an answer to the question in the title of the paper within the OT framework. 

 Defining Hamiltonian process on graph must face several intrinsic difficulties. The most obvious one is that $X(t)$ is a stochastic process jumping from node to node on the graph, while its continuous space counterpart trajectory is a spatial-temporal continuous function. Another challenge is about characteristic line. In fact, it is not clear how to define characteristic on graph.  
Furthermore, there is no reported result about examining whether a stochastic process, such as discrete OT and SBP, can preserve Hamiltonian along its trajectory, just like a classical Hamiltonian system does in continuous space.   

To fill the gaps on finite graph, our idea is lifting the process on graph into a motion on its density manifold. {To be more precise, we define the Hamiltonian process by a random process whose density and generators of instantaneous transition rate matrix form a Wasserstein Hamiltonian flow on the cotangent bundle of density manifold.} Meanwhile, we show that such defined Hamiltonian processes exist in numerous practical problems, such as the discrete OT problem and SBP. Two important classes of Hamiltonian processes, namely the stationary Hamiltonian process and the periodic Hamiltonian process, are also discussed via the framework of SBP. They correspond to the invariant measure and the periodic solution of the Hamiltonian flow on the density space.  {We would like to mention that
the Wasserstein Hamiltonian flow is firstly studied by Nelson's mechanics (see e.g. \cite{Nelson19661079,Carlen1984293}). It is also pointed out that the Hamiltonian flows in density space are probability transition equations of classical Hamiltonian ODEs (see \cite{Vil09,CLZ20} and references therein).}

{There are several works with titles related to Hamiltonian systems on graphs, like the port-Hamiltonian system on graphs (see e.g. \cite{van06,vanM13} and the references therein).} Our current work is different from them.
The port-Hamiltonian systems are the generalization of classical Hamiltonian system which describes the dynamics in interaction with control units, energy dissipating or energy storing units. The graph structure is used to characterize the interaction of the systems with ports, and their underlying phase variables are still in continuous spaces, like $\mathbb R^d$ or smooth manifold. 

This paper is organized as follows. In section 2, we use the discrete optimal transport problem as the motivation of studying the Hamiltonian process on finite graph. In section 3, we present the definition and several properties of the Hamiltonian process on graph. In section 4, we study several different Hamiltonian dynamics derived from the discrete SBP from two different perspectives. We also discuss the existence of stationary and periodic Hamiltonian processes of  the discrete SBP. We provide more examples of Hamiltonian process on graph in section 5.

{\section{Preliminary}}
{
In this section, we first briefly recall the relationship between the continuous OT problem and Hamiltonian systems. Then we introduce our motivation example on a graph and review some notations for inhomogeneous Markov process, which is used in our definition for Hamiltonian process.}

{
It is known that in a continuous OT problem \eqref{min-e}  with given marginal distributions $\rho^0$ and $\rho^1$, the optimal transfer $\{X_t\}_{t\in [0,1]}$ induces a trajectory concentrating on the geodesic path whose position and momentum obey the Hamiltonian principle (see e.g. \cite{Vil09}).} 
More precisely, recalling that $H(x,v)=\frac 12|v|^2$, the critical point of the OT problem \eqref{min-e} in density manifold satisfies the Wasserstein--Hamiltonian flow,
\begin{equation}\label{WassHam1}\begin{split}
\partial_t \rho + \nabla \cdot (\frac {\partial H}{\partial v}(x,\nabla S) \rho)&=0,\\
\partial_t S +H(x,\nabla S)
&= C(t),
\end{split}\end{equation}
where $C(t)$ is a function depending only on $t$ and $v=\nabla S$ with $|\nabla S|^2=\nabla S \cdot \nabla S$. Being a Hamiltonian system on its own,  \eqref{WassHam1} can also be connected to the following classic Hamiltonian system closely (see e.g. \cite{CLZ19}): 
\begin{equation}\label{HamIVP}\begin{split}
d_tv &=-\frac {\partial H}{\partial x}(X,v),\\
d_t X &=\frac {\partial H}{\partial v}(X,v),
\end{split}\end{equation}
where $X \in \mathbb R^d$, the conjugate momenta $v \in \mathbb R^d, d\in \mathbb N^+$, and the Hamiltonian $H$ is smooth. If the initial position $X(0)$ is random following a distribution with density $\rho^0$, the trajectory $X_t$ is random too. Its density function $\rho$, defined by the pushforward operator induced by the $X_t$, $\rho_t=X_t^{\#}\rho^0$, satisfies the Wasserstein-Hamiltonian flow \eqref{WassHam1}. 
However, directly mimicking the relationship between \eqref{WassHam1} and \eqref{HamIVP} is impossible if the underlying space become a graph. In the next subsection, we illustrate the challenges in detail by an example on graph.
 

\subsection{A motivation example}

Consider a graph $G=(V,E, \mathbf W)$ with a node set $V=\{a_i\}_{i=1}^N$, an edge set $E$, and $w_{jl}\in \mathbf W$
are the weights of the edges: 
$w_{jl}=w_{lj}>0$, if there is an edge between $a_j$ and $a_l$, and $0$ otherwise.
Below, we write $(i,j)\in E$ to denote the edge in $E$ between the vertices $a_i$ and $a_j$. We assume that $G$ is an undirected and connected graph with no self loops or multiple edges for simplicity.  Let us denote the set of discrete probabilities on the graph by: 
$$\mathcal P(G)=\{(\rho)_{j=1}^N\ :\, \sum_{j}\rho_j =1, \rho_j\ge 0,\; \text{for} \; j\in V\},$$ 
and let $\mathcal P_o(G)$ be its interior (i.e., all  $\rho_j> 0$, for $a_j\in V$).
Inspired by \cite{CLZ19,CLZ20a,Leo16}, we consider the following discrete OT problem whose minimizer is the so-called geodesic random walk. 

\begin{ex}\label{Geo-rand}
OT on $G$ (geodesic random walk).

The OT problem on a finite graph is related to the Wasserstein distance on $\mathcal{P}(G)$, which 
can be defined by the discrete Benamou--Brenier formula:
\begin{align*}
W(\rho^0,\rho^1):=\inf_{v,\rho}\Big\{\sqrt{\int_{0}^1 \<v,v\>_{\theta(\rho)}}dt \,\ : \,
\frac{d\rho}{dt}+div_G^{\theta}(\rho v)=0, \; \rho(0)=\rho^0,\; \rho(1)=\rho^1\Big\}.
\end{align*} 
where $\rho^0,\rho^1\in \mathcal P(G),$ $\rho\in  H^1([0,1],\mathbb R^N)$ and $v$ is a skew matrix valued function. 
The inner product of two vector fields  $u,v$ is defined by 
$$\<u,v\>_{\theta(\rho)}:=\frac 12\sum_{(j,l)\in E}u_{jl}v_{jl}\theta_{jl}w_{jl}$$
with the weight $\theta_{ij}$ depending on $\rho_i$ and $\rho_j.$
The divergence of the flux function $\rho v$ is defined as 
$$(div_G(\rho v))_j:=-(\sum_{l\in N(j)}w_{jl}v_{jl}\theta_{jl}),$$
where $N(i)=\{a_j\in V: (i,j)\in E\}$ is the adjacency set of node $a_i$.
Then its critical point $(\rho,v)$, with $v=\nabla_G S:=(S_j-S_l)_{(j,l)\in E}$ for some function $S$ on $V$, satisfies 
the following discrete Wasserstein-Hamiltonian flow on the graph $G$:
\begin{equation}\label{dhs}\begin{split}
&\frac {d\rho_i}{d t}+\sum_{j\in N(i)}w_{ij}(S_j-S_i)\theta_{ij}(\rho)=0,\\
&\frac {d S_i}{dt}+\frac 12\sum_{j\in N(i)}w_{ij}(S_i-S_j)^2 \frac {\partial \theta_{ij}(\rho)}{\partial \rho_i}=0.
\end{split}\end{equation}
We may view this equation as a discrete analog of \eqref{WassHam1}. Consequently, its Hamiltonian only consists of the kinetic energy $$\mathcal H(\rho,S)=\frac 14\sum_{ij}(S_i-S_j)^2\theta_{ij}(\rho)w_{ij}.$$

As discussed in \cite{Leo16}, the goal of the discrete OT problem is to find an optimal transport of the informal minimization problem
\begin{align}
\inf_{Q}\Big\{\int_0^T\frac 12\sum_{ij\in E} (v_{ij})^2\theta_{ij}w_{ij}dt: d\rho_t=\rho_t Q_t dt, \;  X(0)\sim \rho_0,\; X(T)\sim \rho_T\Big\},
\end{align} 
where the transition rate matrix $Q_t$ may be written as 
\begin{align*}
(Q_t)_{ii}&=\frac 12\sum_{j\in N(i)}w_{ij}\frac {\theta_{ij}(\rho)}{\rho_i}v_{ij},\\
(Q_t)_{ji}&=-\frac 12 w_{ji}\frac {\theta_{ji}(\rho)}{\rho_j}v_{ji},
\end{align*}
if $\theta_{ij}=\theta_{ji}.$  
In \cite{Leo16}, the minimizer of the above discrete OT problem is called the geodesic random walk which is defined as a random walk whose marginal probability is supported on
the set of geodesic paths on $\mathcal P(G), i.e,$ $X_t$ is determined by the marginal distribution and the instantaneous transition rate matrix $Q_t$.
However, examining the transition rate matrix, we can find that the geodesic random walk $X_t$ may not be well-defined, because there may not exist such a stochastic process due to possible negative probability and transition probability (See Remark \ref{rk-ot} for more details).
\end{ex}

This example illustrates that when compared to the continuous case, where the Hamiltonian system \eqref{HamIVP} on the phase space corresponds to the Hamiltonian PDEs \eqref{WassHam1} on Wasserstein manifold, such a correspondence in discrete space can't be easily established, because the counterpart of \eqref{HamIVP} requires more careful treatments.  

\subsection{Inhomogeneous Markov process}

In order to define a stochastic process which plays the role of the Hamiltonian mechanics \eqref{HamIVP} on a finite graph, we recall the definition of the inhomogeneous Markov process in \cite{Kol10}.
The linear master equation 
\begin{align*}
\frac {d\rho}{dt}=\rho Q
\end{align*}
determines a linear Markov process.
When $Q=Q(t),$ it corresponds to a time inhomogeneous Markov process.
Here $Q(t)$ is a family of infinitesimal generators of the stochastic matrix or Kolmogorov matrix, namely, a square matrix which has non-positive (resp. non-negative) elements on the main diagonal (resp. off the main diagonal), and the sum of each row is zero. Among different types of inhomogeneous Markov process, the nonlinear Markov processes whose transition function may depend not only on the current state of the process but also on the current distribution of the process is of particular interest to us.

Given an initial distribution $\rho_0$, a time inhomogeneous Markov process $\{X_t\}_{t\ge 0}$ can be defined as a process which has $\rho_0$ as the distribution of $X_0$ and $(s,t)\to P_{s,t}$ as its transition mechanism in the sense that 
\begin{align*}
\mathbb P(X_0=a_i)=\rho_i,\;
\mathbb P(X_t=a_j | X_{\sigma},\sigma\in [0,s])
=(P_{s,t})_{X(s)a_j},
\end{align*} 
where $(P_{s,t})_{a_ia_j}=\mathbb P(X_t=a_j|X_s=a_i).$
The corresponding forward Kolmogorov equation can be rewritten as 
\begin{align*}
d_tP_{s,t}=P_{s,t}Q_t.
\end{align*}
If $t\in [s,\infty)\mapsto \rho_t$ is continuously differentiable,
then 
\begin{align*}
\dot \rho_t=\rho_t Q_t, 
\end{align*}
is equivalent to
\begin{align*}
\rho_t=\rho_s P_{s,t},
\end{align*}
for $t\ge s.$
Given $(Q_t)_{t\ge 0}$, $\rho_0,$ there exists an inhomogeneous  Markov process $X_t$ related to the transition rate matrix $Q_t$ and the marginal distribution $\rho_t.$ On the other hand, given an inhomogeneous Markov process with transition matrices $P_{s,t}$, it will induce the equation of $\rho$ with $Q_t$ (see e.g. \cite{Kol10}).

\section{Hamiltonian process on a finite graph}

As shown in the Example \ref{Geo-rand}, although it may not be possible to find a stochastic process for every discrete optimal transport problem, it reveals  
two key features that the density of such a stochastic process, if exists, satisfies the generalized master equation and that its $Q
_t$-matrix is determined by a potential $S_t$, where $S_t$ satisfies a discrete Hamiltonian Jacobi equation. Inspired by these properties, we introduce the definition of stochastic Hamiltonian process.

 \begin{df}\label{df-hm-2}
 A stochastic process $\{X_t\}_{t\ge 0}$ is called a Hamiltonian process on the graph if 
 \begin{enumerate}
 \item The density $\rho_t$ of $X_t$ satisfies the following generalized Master equation,
 \begin{align*}
 d_t\rho_t=\rho_t Q_t,
 \end{align*}
with $ (Q_t)_{ij}=w_{ji}f_{ji}(v_{ji}), (Q_t)_{ii}=-\sum_{j\in N(i)}w_{ij}f_{ji}(v_{ji}),$ where the skew-matrix $v$ is induced by a potential function $S$, i.e. $v=\nabla_G S+u,$ with $div_G(\rho u)=0.$
And $f_{ij}: \mathbb R\to\mathbb R, $ is a real-valued measurable function which is piecewise continuous in $x\in \mathbb R$ and may depend on $t$ and $\rho$.  
 \item The density $\rho$ and the potential $S$ form a Hamiltonian system on the cotangent bundle of the density space. 
 \end{enumerate}
 \end{df}
 
 The following theorem gives the structure of the Hamiltonian on the density manifold of the Hamiltonian process.

 \begin{tm}
   Suppose that the stochastic process $\{X_t\}_{t\geq 0}$ with density $\{\rho_t\}_{t \geq 0}$ and potential $\{S_t\}_{t\geq 0}$ defined in the Definition \ref{df-hm-2} forms a Hamiltonian process on the graph $G$. 
  In addition assume that the antiderivative $F_{ij}$ of $f_{ij}$ exists for $ij\in E.$
   Then the Hamiltonian always have the form 
   \begin{align}
    \mathcal{H}(\rho, S) & = \sum_{i\in V}\sum_{j\in N(i)} \rho_i F_{ji}(S_j-S_i,\rho,t)w_{ji} + \mathcal{V}(\rho, t) \label{hamiltonian} 
   \end{align}
   where $\mathcal V$ is a function depending $\rho$ and $t.$
Moreover, the Hamiltonian system on the cotangent bundle of $\mathcal{P}(G)$ can be written as:
  \begin{align*}
    & \frac{\partial}{\partial t} \rho_i(t) = \sum_{j\in N(i)} w_{ij}f_{ij}(S_i-S_j, \rho, t)\rho_j - w_{ji}f_{ji}(S_j-S_i, \rho, t)\rho_i \\
    & \frac{\partial}{\partial t} S_i(t) = - \sum_{j\in N(i)} \left( w_{ji}F_{ji}( S_j - S_i ,\rho, t ) + \rho_i \frac{\partial}{\partial \rho_i}F_{ji}(S_j-S_i, \rho, t)w_{ji} \right) - \frac{\partial}{\partial \rho_i} \mathcal{V}(\rho, t).
  \end{align*}
\end{tm}

\begin{proof}
According to Definition \ref{df-hm-2}, we have $\frac{\partial}{\partial t} \rho_i(t) = \sum_{j\in N(i)} \omega_{ij}f_{ij}(S_i-S_j, \rho, t)\rho_j - w_{ji}f_{ji}(S_j-S_i, \rho, t)\rho_i $.  Since $\{\rho_t, S_t\}$ forms a Hamiltonian system, we are able to state 
\begin{equation*}
   \frac{\partial}{\partial S_i}\mathcal{H}(\rho, S, t ) =  \sum_{j\in N(i)} w_{ij}f_{ij}(S_i-S_j,\rho, t)\rho_j - w_{ji}f_{ji}(S_j-S_i,\rho,t)\rho_i,   \quad i\in V.
\end{equation*}
Considering the following quantity,
\begin{equation*}
  \mathcal{H}_0(\rho, S, t) =  \sum_{i\in V}\sum_{j\in N(i)} \rho_i F_{ji}(S_j-S_i,\rho,t)w_{ji}, 
\end{equation*}
 we can directly verify that $\frac{\partial}{\partial S}(\mathcal{H}(\rho, S, t)-\mathcal{H}_0(\rho, S , t))=0$. This suggests that there exists some function $\mathcal V$ depending on $\rho$ and $t$ such that $\mathcal{H}(\rho, S, t ) - \mathcal{H}_0(\rho, S, t) = \mathcal{V}(\rho, t)$. This directly leads to form of Hamiltonian $\mathcal{H}(\rho, S, t) = \sum_{i\in V}\sum_{j\in N(i)} \rho_i F_{ji}(S_j-S_i, \rho, t )w_{ji} + \mathcal{V}(\rho, t ) $. 
Furthermore, the discrete Hamiltonian Jacobi equation is derived as
\begin{equation*}
  \frac{\partial}{\partial t}S_t = - \frac{\partial}{\partial \rho}\mathcal{H}(\rho, S, t).
\end{equation*}
\end{proof}

As a direct consequence, we have the following properties of Hamiltonian process.

\begin{prop}
Assume that a stochastic process $X_t$ on a finite graph is a Hamiltonian process. Then it holds that 
 \begin{enumerate}
 \item there exists a Hamiltonian $\mathcal H$ on the density space such that its marginal distribution $\rho_t= X_t^{\#}\rho_0$ and the generator $S_t$ of the transition rate matrix $Q_t$ forms a Hamiltonian system;
  
 \item the symplectic structure on the density space is preserved, i.e., 
 \begin{align*}
 \omega_{g(\rho,S)}(g'(\rho,S)\xi,g'(\rho,S)\eta)=\omega_{(\rho,S)}(\xi,\eta),
 \end{align*}
 where $\xi,\eta\in T_{\rho,S}(T^{*}P(G)),$ and $g'(\rho,S)$ is the Jacobi matrix of the Hamiltonian flow on the density space;
 \item $\mathcal H(t)=\mathcal H(0)$, if the Hamiltonian $\mathcal H$ is independent of $t$; 
 \item and $X_t$ is mass-preserving, i.e., $\sum_{i=1}^N\rho_i(t)=\sum_{i=1}^N\rho_i(0).$ 
 \end{enumerate}
\end{prop}

Based on the definition of Hamiltonian process, we are able to construct the discrete optimal transport problem which retains the property that the minimizer is a stochastic process on the graph for Example \ref{Geo-rand}.

\begin{prop}
There always exists a density dependent weight $\theta$ such that the geodesic random walk in Example \ref{Geo-rand} is a Hamiltonian process.
\end{prop}

\begin{proof}
Define $\theta^U_{ij}=\theta^U_S (\rho_i,\rho_j),$ where $\theta^U_S(\rho_i,\rho_j)=\rho_i$ if $S_j>S_i.$
Denote $(x)^+=\max(0,x), (x)^{-}=\min(0,x).$
Using the notations in Example \ref{Geo-rand}, the geodesic random walk on $G$ with the probability weight $\theta=\theta^U$ satisfies  
\begin{align}\label{geo-den}
d\rho_i&=\sum_{j\in N(i)}w_{ij}(v_{ij})^+\rho_j+\sum_{j\in N(i)}w_{ij}(v_{ij})^-\rho_i.
\end{align}
From the discrete Hodge decomposition on the graph \cite{CLZ19}, for any skew matrix $v$ and probability density $\rho\in \mathcal P_o(G)$, there exists a decomposition $v=\nabla_G S+u$ with $div_G(\rho u)=0.$
Here $(\nabla_G S)_{ij}:=(S_i-S_j)$ and $div_G(\rho u):=-(\sum_{j\in N(i)} u_{ij}\theta_{ij}^U(\rho)).$
To see this fact, it suffices to prove that there exists a unique solution of $S$ such that $div_G(\rho\nabla_G S)=div_G(\rho v).$ The connectivity of the graph and the fact that $\rho\in \mathcal P_o(G)$ implies that if 
\begin{align*}
\<div_G(\rho\nabla_G S),S\>=\frac 12\sum_{(i,j)\in E}((S_i-S_j)^-)^2\theta_{ij}(\rho)=0,
\end{align*}
then $0$ must be a simple eigenvalue of $div_G(\rho \nabla_G)$ with eigenvector $(1,\cdots,1).$ Thus $S$ is unique up to a constant shift and 
the skew matrix $v_t=\nabla_G S_t+u$ satisfies  
\begin{align*}
d(S_t)_{i}=-\frac 12\sum_{j\in N(i)}w_{ij}((S_{i}-S_{j})^-)^2+C(t),\;  div_G(\rho u)=0,
\end{align*}
where $C(t)$ is independent of nodes.  Meanwhile, $f_{ij}$ can be selected to achieve 
$f_{ij}(S_i-S_j)=(S_i-S_j)^+$  and thus
\begin{align*}
(Q_t)_{ii}&=\sum_{j\in N(i)}w_{ij}(S_i-S_j)^{-}=\sum_{j\in N(i)}w_{ij}f_{ji}(S_j-S_i),\\
(Q_t)_{ji}&=w_{ji}(S_j-S_i)^+=w_{ji}f_{ji}(S_j-S_i),\; ij\in E ,\; \text{otherwise} \; Q_{ji}=0.
\end{align*} 

We can define a time inhomogenous Markov process as follows by the
transition matrix $\mathbb P(X_t=v_j | X_{\tau},\tau\in [0,s])
=(P_{s,t})_{X(s)v_j}$.
Given the past $\sigma(\{X_{\tau}:\tau\in [0,t]\})$ of $X$ up to time $t\ge0,$ the probability of its having moved away from $X_t$ at the time $t+h$
with $h$ small enough can be approximated by $1-(Q_t)_{X_tX_t}h,$ i.e.,
\begin{align*}
\Big|\mathbb P(X(t+h)=X_t | X_{\tau}, \tau\le t)-1-(Q_t)_{X_tX_t}h\Big|= o(h).
\end{align*}
 Here $\{-(Q_t)_{ii}\}_{i}$ is often called as the transition rate of $X_t.$
Given the history that the jump appeared $\sigma(\{X_{\tau}:\tau\in [0,t] \}\cup\{X_{t+h}\neq X_t\})$, the probability that $X_{t+h}=a_j$ is approximately $(P_{t,t+h})_{X_ta_j}$, which implies that   
\begin{align*}
\Big|\mathbb P(X(t+h)=a_j|X_{\tau},\tau\le t)-h(Q_t)_{X_ta_j}\Big|= o(h).
\end{align*}
\end{proof}

\begin{rk}\label{rk-ot}
Not every discrete optimal transport problem on the graph  has a corresponding stochastic process as a realization. The density dependent weight in the discrete optimal transport problem can not be chosen arbitrarily. Take $w_{ij}=1$ if $ij\in E$ for simplicity.
For example, if we take the probability weight $\theta_{ij}=\theta^A(\rho_i,\rho_j)=\frac 12(\rho_i+\rho_j)$ in \cite{CLZ19}, the density equation can be rewritten as 
\begin{align*}
d_t\rho_t=\rho_t Q_t,
\end{align*}
where 
\begin{align*}
(Q_t)_{ii}&=\frac 12\sum_{j\in N(i)}(S_i-S_j),\\
(Q_t)_{ij}&=\frac 12(S_j-S_i),\; ij\in E ,\; \text{otherwise} \; Q_{ij}=0.
\end{align*}
The function $f_{ij}(x)=\frac 12 x$.
When $\theta_{ij}=\theta^L(\rho_i,\rho_j)=\frac {\rho_i-\rho_j}{\log(\rho_i)-\log(\rho_j)}$ in \cite{CHLZ12}, the density equation can be rewritten as 
\begin{align*}
d_t\rho_t=\rho_t Q_t,
\end{align*}
where 
\begin{align*}
(Q_t)_{ii}&=\sum_{j\in N(i)}\frac{(S_i-S_j)}{\log(\rho_i)-\log(\rho_j)},\\
(Q_t)_{ij}&=-\frac{(S_i-S_j)}{\log(\rho_i)-\log(\rho_j)},\; ij\in E ,\; \text{otherwise} \; Q_{ij}=0.
\end{align*}
The function $f_{ij}(x)=\frac x{\log(\rho_i)-\log(\rho_j)}$.
In these two cases, there is no guarantee that the off-diagonal of $Q_t$ is non-positive. If this happens, $Q_t$ does not determine a process $X_t$ which is time inhomogeneous Markov since the negative probability and the negative transition probability may appear. For more possible choices of $\theta,$ we refer to \cite{Mas11} and references therein.
\end{rk}

 \begin{rk}
If $\theta_{ij}>0$ for all $ij\in E$, then the Hodge decomposition yield a unique potential $S$ up to a constant which induces $v.$ If there exists $ij\in E$ such that $\theta_{ij}=0$, then the generator $S$ may be not unique.
Meanwhile, the Hamiltonian Jacobi equation may become one-side inequality,
\begin{align*}
v_{ij}=S_i-S_j, \;
\partial_t S_i+  \frac {\partial}{\partial \rho_i} \mathcal{H}(\rho, S) \le  0.
\end{align*}
\end{rk}

\begin{rk}
The initial value problem of the Hamiltonian system of $\rho,S$ may develop singularity at a finite time $T>0$, i.e, either $\lim_{t\to T}S_i(t)=\infty$ or $\lim_{t\to T}\rho_i\le 0.$
 \end{rk}

We would like to emphasize that a Hamiltonian process is not Markov in general.
The sufficient and necessary conditions when a Hamiltonian process gives a Markov process is presented as follows.
 
 \begin{tm}\label{mar-ham}
Given a Hamiltonian process $\{X_t\}_{t\ge 0}$ on the graph with a Hamiltonian $\mathcal H(\rho,S)=\sum_{i=1}^N\sum_{j\in N(i)}F_{ij}(\rho,S)w_{ij}\rho_i$. 
If $X_t$ is a Markov process, then 
$(\rho,S)$ in Definition \eqref{df-hm-2} satisfies the following system,
 \begin{align}\label{mar-con}
 &\frac {\partial^2 F_{ij}}{\partial S_i\partial \rho_i}\Big(\sum_{l\in N(i)}\frac {\partial F_{il}}{\partial S_i}\rho_iw_{il}+\sum_{l\in N(i)}\frac {\partial F_{li}}{\partial S_i}\rho_lw_{li}\Big)\\\nonumber
&+\frac {\partial^2 F_{ij}}{\partial S_i\partial \rho_j}\Big(\sum_{k\in N(j)}\frac {\partial F_{jk}}{\partial S_j}\rho_jw_{jk}+\sum_{k\in N(j)}\frac {\partial F_{kj}}{\partial S_j}\rho_lw_{kj}\Big)\\\nonumber 
&-\frac {\partial^2 F_{ij}}{\partial S_i\partial S_i}\Big(\sum_{l\in N(i)}\frac {\partial F_{il}}{\partial \rho_i}\rho_iw_{il}+\sum_{l\in N(i)}\frac {\partial F_{li}}{\partial \rho_i}\rho_lw_{li}+\sum_{l\in N(i)}(F_{il}w_{il}+F_{li}w_{li})\Big)
\\\nonumber 
&-\frac {\partial^2 F_{ij}}{\partial S_i\partial S_j}\Big(\sum_{k\in N(j)}\frac {\partial F_{jk}}{\partial \rho_j}\rho_jw_{jk}+\sum_{k\in N(j)}\frac {\partial F_{kj}}{\partial \rho_j}\rho_kw_{kj}+\sum_{k\in N(j)}(F_{jk}w_{jk}+F_{ki}w_{ki})\Big)=0
 \end{align}
for $i,j\in V.$
Conversely, if $(\rho,S)$ satisfies \eqref{mar-con}, then there exists a Markov process which is Hamiltonian.

\end{tm}

\begin{proof}
Since $X_t$ is a Hamiltonian process, the transition matrix is determined by 
$\rho_tQ_t=\frac {\partial H}{\partial S}=d_t\rho_t.$
This implies that 
\begin{align*}
(\rho_t Q_t)_{i}=\sum_{j\in N(i)} \frac {\partial F_{ij}(\rho,S)}{\partial S_i}w_{ij}\rho_i
+\sum_{j\in N(i)} \frac {\partial F_{ji}(\rho,S)}{\partial S_i}w_{ji}\rho_j.
\end{align*}
Therefore, $(Q_t)_{ii}=\sum_{j\in N(i)} \frac {\partial F_{ij}(\rho,S)}{\partial S_i}w_{ij}$, $(Q_t)_{ij}=\frac {\partial F_{ij}(\rho,S)}{\partial S_j}w_{ij}.$
Since $X_t$ preserves the mass, it holds that $\sum_{j\in N(i)} (\frac {\partial F_{ij}(\rho,S)}{\partial S_i}+\frac {\partial F_{ij}(\rho,S)}{\partial S_j})w_{ij}=0$ for every $i\le N.$

Notice that $X_t$ is Markov implies that $d_tQ_{ij}=0,$ for $i,j\le N,$ that is
\begin{align*}
d_t \frac {\partial F_{ij}}{\partial S_i}=0, d_t\frac {\partial F_{ji}}{\partial S_j}=0.
\end{align*}
Direct calculation leads to 
\begin{align*}
d_t \frac {\partial F_{ij}}{\partial S_i}
&=\frac {\partial^2 F_{ij}}{\partial S_i\partial \rho_i}d_t\rho_i
+\frac {\partial^2 F_{ij}}{\partial S_i\partial \rho_j}d_t\rho_j\\
&+\frac {\partial^2 F_{ij}}{\partial S_i\partial S_i}d_tS_i
+\frac {\partial^2 F_{ij}}{\partial S_i\partial S_j}d_tS_j\\
&=\frac {\partial^2 F_{ij}}{\partial S_i\partial \rho_i}\Big(\sum_{l\in N(i)}\frac {\partial F_{il}}{\partial S_i}\rho_iw_{il}+\sum_{l\in N(i)}\frac {\partial F_{li}}{\partial S_i}\rho_lw_{li}\Big)\\
&+\frac {\partial^2 F_{ij}}{\partial S_i\partial \rho_j}\Big(\sum_{k\in N(j)}\frac {\partial F_{jk}}{\partial S_j}\rho_jw_{jk}+\sum_{k\in N(j)}\frac {\partial F_{kj}}{\partial S_j}\rho_lw_{kj}\Big)\\
&-\frac {\partial^2 F_{ij}}{\partial S_i\partial S_i}\Big(\sum_{l\in N(i)}\frac {\partial F_{il}}{\partial \rho_i}\rho_iw_{il}+\sum_{l\in N(i)}\frac {\partial F_{li}}{\partial \rho_i}\rho_lw_{li}+\sum_{l\in N(i)}(F_{il}w_{il}+F_{li}w_{li})\Big)
\\
&-\frac {\partial^2 F_{ij}}{\partial S_i\partial S_j}\Big(\sum_{k\in N(j)}\frac {\partial F_{jk}}{\partial \rho_j}\rho_jw_{jk}+\sum_{k\in N(j)}\frac {\partial F_{kj}}{\partial \rho_j}\rho_kw_{kj}+\sum_{k\in N(j)}(F_{jk}w_{jk}+F_{ki}w_{ki})\Big),
\end{align*}
which yields the desired result.
Conversely, if $(\rho,S)$ satisfies \eqref{mar-con}, the previous arguments leads to the equation of $\rho$ becomes a linear Master equation. Then there always exists a Markov process which is a stochastic representation of linear Master equation.  Meanwhile, it can be verified that this Markov process satisfies all the conditions in Definition \ref{df-hm-2} and is Hamiltonian.
\end{proof}

\begin{cor}
Given a Hamiltonian $\mathcal H(\rho,S)=\sum_{i=1}^N\sum_{j\in N(i)}F_{ij}(\rho,S)w_{ij}\rho_i$. Assume that there exists $(\rho^*,S^*(t))$ satisfies the following conditions,
\begin{enumerate} 
 \item $\sum_{j\in N(i)}\frac {\partial F_{ij}(\rho,S)}{\partial S_i}+\frac {\partial F_{ij}(\rho,S)}{\partial S_j}=0,$ 
 \item  $\rho^*$ is independent of $t$ and $(\rho^*,S^*(t))$ solves
 \begin{align*}
 &\sum_{l\in N(i)}\frac {\partial F_{il}}{\partial S_i}\rho_iw_{il}+\sum_{l\in N(i)}\frac {\partial F_{li}}{\partial S_i}\rho_lw_{li}=0,\\
 &\frac {\partial^2 F_{ij}}{\partial S_i\partial S_i}\Big(\sum_{l\in N(i)}\frac {\partial F_{il}}{\partial \rho_i}\rho_iw_{il}+\sum_{l\in N(i)}\frac {\partial F_{li}}{\partial \rho_i}\rho_lw_{li}+\sum_{l\in N(i)}(F_{il}w_{il}+F_{li}w_{li})\Big)
\\
&+\frac {\partial^2 F_{ij}}{\partial S_i\partial S_j}\Big(\sum_{k\in N(j)}\frac {\partial F_{jk}}{\partial \rho_j}\rho_jw_{jk}+\sum_{k\in N(j)}\frac {\partial F_{kj}}{\partial \rho_j}\rho_kw_{kj}+\sum_{k\in N(j)}(F_{jk}w_{jk}+F_{ki}w_{ki})\Big)=0
\end{align*}
 \end{enumerate}
Then there exists a Hamiltonian process which is Markov and preserves the mass. Furthermore, the Hamiltonian process is invariant with respect to $\rho^*.$ 

\end{cor}

\section{Hamiltonian process via discrete SBP on graphs}
Although the SBP \cite{Sch31} has a history close to 100 years, it has received revived attention from control theory
and machine learning communities recently, see \cite{Leo14,Pav03}.
For convenience, the background of continuous SBP is presented in the appendix.

For the discrete counterpart of SBP  on graph,
roughly speaking, there are two different treatments reported by existing references. 
\begin{itemize}
\item One is to consider a reference path measure (induced by a reversible random walk) on the graph and then study the optimization problem involving the relative entropy between the reference measure and the path measure with given initial and terminal distributions \cite{Leo14, Leo16}.)  
In this framework, the reference random walk is often related to a discrete version of \eqref{heat} (For example, the linear discretization of the Laplacian gives the time homogenous Markov chain as the reference in \cite{CGPT17}).
\item Another way is proposed by the discrete version of \eqref{schro1} or \eqref{schro0} directly \cite{chow20discrete}. 
\end{itemize}
We shall show that different treatments create differences on the structure and formulation of equations, in particular the discrete Laplacian operator. Each of these formulations can determine its corresponding Hamiltonian process on graph.

\subsection{Discrete SBP based on relative entropy and reference Markov measure}

 In the following discussion, we always assume that $w_{ij}=1$ if $ij\in E$ for conciseness of formulations. By using the discrete Girsanov theorem on graph, the discrete SBP in the form of relative entropy (A) becomes the following control problem 
\begin{align} 
  & \min_{\widehat{m}^t\geq 0}  \left\{ \int_0^1 \sum_{i\in V} \rho(i,t)  \sum_{j\in N(i)} \left(\frac{\widehat{m}^t_{ij}}{m_{ij}^t}\log\left(\frac{\widehat{m}^t_{ij}}{m_{ij}^t}\right) -\frac{\widehat{m}^t_{ij}}{m_{ij}^t} + 1\right)m_{ij}^t ~dt \right\}  \label{SBP_KL}\\
  & \textrm{subject to:} \quad \frac{d}{dt} \rho(i,t) = \sum_{j\in N(i)} \widehat{m}^t_{ji}\rho_j - \widehat{m}^t_{ij}\rho_i \quad \rho(\cdot,0)=\rho^0, ~\rho(\cdot,1)=\rho^1. \nonumber   
\end{align}
where the reference measure $R$ is determined by the master equation $d_t\widetilde \rho_t=\sum_{j\in N(i)} {m}^t_{ji}\widetilde \rho_j - {m}^t_{ij}\widetilde \rho_i $. We provide a detailed derivation of \eqref{SBP_KL} in the appendix.
Let us denote $u(x) = x\log x-x+1$.
By introducing Lagrange multiplier $\psi$, we obtain the following Lagrangian functional
\begin{align*} 
  & \mathcal{L}(\rho,\widehat{m},\psi) = \int_0^1 \sum_{i\in V}\rho(i,t)  \sum_{j\in N(i)} u\left(\frac{\widehat{m}^t_{ij}}{m_{ij}^t}\right) m_{ij}^t~dt \\ & +\int_0^1 \sum_{i\in V} - \rho(i,t)\frac{\partial}{\partial t}\psi(i,t) -\psi(i,t) \left(\sum_{j\in N(i)} \widehat{m}^t_{ji}\rho_j - \widehat{m}^t_{ij}\rho_i\right) ~dt \\
  = & \int_0^1 - \sum_{i\in V}  \rho(i,t)\frac{\partial}{\partial t}\psi(i,t) -\frac 12\sum_{(i,j)\in E}\left[\frac{\widehat{m}_{ji}}{m_{ji}}(\psi(i,t)-\psi(j,t))-u(\frac{\widehat{m}_{ji}}{m_{ji}})\right]m_{ji}\rho(j,t)\\
   &+\left[\frac{\widehat{m}_{ij}}{m_{ij}}(\psi(j,t)-\psi(i,t))-u(\frac{\widehat{m}_{ij}}{m_{ij}})\right]m_{ij}\rho(i,t) ~dt.
\end{align*}
When solving the above saddle point problem, we minimize over $\widehat{m}$ and get 
\begin{align*}
  &\int_0^1- \sum_{i\in V}  \rho(i,t)\frac{\partial}{\partial t}\psi(i,t) -\frac 12\sum_{(i,j)\in E}[u^*(\psi(i,t)-\psi(j,t))m_{ji}\rho(j,t) \\
  &\quad + u^*(\psi(j,t)-\psi(i,t))m_{ij}\rho(i,t)]~dt.
\end{align*}
Here $u^*$ is the Legendre dual of $u$: $u^*(x) = \sup_y\left\{x\cdot y-u(y)\right\}$, leading to $u^*(x)=e^x-1$.
By formulating the Lagrangian, we can identify the Hamiltonian of this control problem, which can be written as:
\begin{equation}
  \mathcal{H}(\rho,\psi) = \sum_{i\in V}\sum_{j\in N(i)} \exp(\psi(j,t)-\psi(i,t))m_{ij}\rho(i,t).       \label{hamiltonian SBP}
\end{equation}
Then the above control problem implies the following Hamiltonian system
\begin{equation*}
  \partial_t \rho = \frac{\partial \mathcal H(\rho,\psi)}{\partial\psi}, \quad \partial_t\psi = -\frac{\partial \mathcal H(\rho,\psi)}{\partial\rho},
\end{equation*}
that is,
\begin{align}
 & \frac{\partial }{\partial t} \rho(i,t) = \sum_{j\in N(i)} -e^{\psi(j,t)-\psi(i,t)}m_{ij}\rho(i,t) +e^{\psi(i,t)-\psi(j,t)}m_{ji}\rho(j,t), \label{ham-schr1} \\
 & \frac{\partial }{\partial t} \psi(i,t) = -\sum_{j\in N(i)}(e^{\psi(j,t)-\psi(i,t)}-1)m_{ij}.\nonumber
\end{align}
By using the Hopf-Cole transform, we can further verify that the discrete SBP problem determines a Hamiltonian process on the graph. 
Let us consider the following transform $\tau:T^*\mathcal{P}(G)\rightarrow T^*\mathcal{P}(G)$ as:
\begin{equation}
  \tau[(\rho,\psi)] = (\rho, \ln(\psi)-\frac{1}{2}\ln\rho)  \quad \label{symplectic transf as hopfcole}
\end{equation}
Let us denote  $g'(\rho,\psi) = D\tau(\rho,\psi)$.
Then the symplectic form $\omega$ is unchanged in the sense that 
\begin{align*}
 \omega_{g(\rho,\psi)}(g'(\rho,\psi)\xi, g'(\rho,\psi)\eta)=\omega_{(\rho,\psi)}(\xi,\eta),
\end{align*}
where $(\xi,\eta)\in T_{(\rho,\psi)} T^*\mathcal P(G).$
By using the symplectic submersion from $\mathcal P(G)$ to $\mathbb R^N$, the symplectic form can be represented by $(g'(\rho,\psi)\xi)^TJg'(\rho,\psi)\eta=\xi^TJ\eta,$ where $J$ is the standard symplectic matrix. 
Since $d_t\tau(\rho,\psi)^T= \tau' d_t(\rho, \psi)^T$ and that $(\tau')^TJ \tau'=J$, we conclude that the Hopf--Cole transformation \eqref{symplectic transf as hopfcole} is a symplectic transformation on the cotangent bundle of the density manifold. 
Denote $(\rho,S)$ as the new coordinate. Then $\{\rho_t,S_t\}$ satisfies the following Hamiltonian system:
\begin{align*}
 &\frac{\partial\rho(i,t)}{\partial t} = \frac{\partial \tilde{\mathcal H}(\rho,S)}{\partial S} \\
 & \frac{\partial S(i,t)}{\partial t} =  - \frac{\partial \tilde{\mathcal H}(\rho,S)}{\partial \rho}
\end{align*}
with 
\begin{equation}
  \tilde{\mathcal{H}}(\rho,S) = \mathcal{H}(\tau^{-1}(\rho, S))
 = \sum_{i\in V}\sum_{j\in N(i)} e^{(S_j-S_i)} m_{ij}\sqrt{\rho_i\rho_j},
\end{equation}
that is
\begin{align}
\frac{\partial S(i,t)}{\partial t}
&=-m_{ii}-\frac 12 \sum_{j\in N(i)}e^{S_j-S_i}m_{ij}\frac {\sqrt{\rho_j}}{\sqrt{\rho_i}}-\frac 12\sum_{j\in N(i)}e^{S_i-S_j}m_{ji}\frac {\sqrt{\rho_j}}{\sqrt{\rho_i}},\label{ham-schr}\\
\frac{\partial \rho(i,t)}{\partial t}&=\sum_{j\in N(i)}e^{S_i-S_j}m_{ji}\sqrt{\rho_j}\sqrt{\rho_i}-\sum_{j\in N(i)}e^{S_j-S_i}m_{ij}\sqrt{\rho_i}\sqrt{\rho_j}.\nonumber
\end{align}

As a consequence, we verify that, as reported in \cite{Leo14} 
the discrete SBP corresponds to a Hamiltonian process with the transition rate matrix $Q$ $(Q_{ij}=\widehat m_{ij})$ defined by $Q_{ii}=-\sum_{j\in N(i)}e^{S_j-S_i}\frac{\sqrt{\rho_j}}{\sqrt{\rho_i}} m_{ij},$ $ Q_{ij}=e^{S_i-S_j}\frac{\sqrt{\rho_i}}{\sqrt{\rho_j}}m_{ji}$ if $ij\in E$. 

Using the above procedures, we can naturally extend the original SBP problem to the following generalized control problem

\begin{align}\label{gen-sch}
  & \min_{\hat{m}^t\geq 0}  \left\{ \int_0^1 \sum_{i\in V} \rho(i,t)  \sum_{j\in N(i)} u\left(\frac{\hat{m}^t_{ij}}{m^t_{ij}}\right)m_{ij}^t ~dt \right\}  \\
  & \textrm{subject to:} \quad \frac{d}{dt} \rho(i,t) = \sum_{j\in N(i)} \hat{m}^t_{ji}\rho_j - \hat{m}^t_{ij}\rho_i \quad \rho(\cdot,0)=\rho^0, ~\rho(\cdot,1)=\rho^1.  \nonumber 
\end{align}

Here $u$ is an arbitrary convex function. Then the Hamiltonian associated  with this general control problem is
\begin{equation}
  \mathcal{H}(\rho,\psi) = \sum_{i\in V} \sum_{j\in N(i)} u^*(\psi(j,t)-\psi(i,t))m_{ij}\rho(i,t), \label{hamiltonian general}
\end{equation}
where $\lambda_{ij}=(u')^{-1}(\psi_j-\psi_i).$

For the sake of completeness of our discussion, we also reveal the relations among the so-called Schr\"{o}dinger system \cite{CLZ20, conforti2017extremal, Blaquire1992ControllabilityOA} and our derived systems \eqref{ham-schr1} and \eqref{ham-schr}. All three PDE systems are derived from the SBP. We introduce the Madelung Transform  $ \phi:T^*\mathcal{P}(G)\rightarrow T^*\mathcal{P}(G)$
\begin{equation}
  (f,g) = \phi(\rho, S) = (\sqrt{\rho}e^{-S}, \sqrt{\rho}e^S) \label{Madelung1}
\end{equation}
Or equivalently,
\begin{equation}
 (f,g)  = \tilde{\phi}(\rho, \psi) = (\rho e^{-\psi}, e^\psi)\label{Madelung2}
\end{equation}
Combining \eqref{Madelung1} with \eqref{ham-schr}, or combining \eqref{Madelung2} with \eqref{ham-schr1} yields the Schr\"{o}dinger system:
\begin{align}
  & \frac{\partial}{\partial t} f(i,t) = \sum_{j\in N(i)} (f(j,t)-f(i,t))m_{ij}^t, \label{ham-schr3} \\
  & \frac{\partial}{\partial t }g(i,t) = -\sum_{j\in N(i)} (g(j,t)-g(i,t))m_{ij}^t.\nonumber
\end{align}
Similar to our previous analysis, we can verify that both transforms $\phi$ and $\tilde{\phi}$ preserves the symplectic form. And we know that \eqref{ham-schr3} is a Hamiltonian system and its corresponding Hamiltonian is
\begin{equation*}
  \widehat{\mathcal{H}}(f,g) = \sum_{i\in V}\sum_{j\in N(i)} f_ig_jm^t_{ij}.
\end{equation*}

By applying the Theorem \ref{mar-ham}, we obtain the following result about the conditions under which the Hamiltonian process in SBP enjoys the stationary measure and Markov property.

\begin{prop}\label{prop-sta}
Assume that the reference process is mass-preserving, i.e.,  {$\sum_{i}\widetilde \rho(i,t)=\sum_{i}\widetilde \rho^0(i)$}, and possesses a stationary measure $\rho^*.$
Then there exists a stationary point $(\rho^*,S^*)$ of the Hamiltonian system \eqref{ham-schr} on the density manifold.
 \end{prop}

\begin{proof}
Take $\frac {\partial \tilde H} {\partial S}=0$ and $\frac {\partial \tilde H}{\partial \rho}=0$ such that $(\rho,S)$ is independent of time. 
The equation of $\rho$ leads to 
\begin{align*}
\sum_{j\in N(i)}e^{S_i-S_j}m_{ji}\sqrt{\rho_j}=\sum_{j\in N(i)}e^{S_j-S_i}m_{ij}\sqrt{\rho_j}.
\end{align*}
Due to $m_{ii}=-\sum_{j\in N(i)}m_{ij},$
the equation of $S$ becomes
\begin{align*}
\frac 12\sum_{j\in N(i)}(e^{S_i-S_j}m_{ji}+e^{S_j-S_i}m_{ij})\sqrt{\rho_j}=\sum_{j\in N(i)}{m_{ij}}\sqrt{\rho_i}.
\end{align*}
Applying  the above relationships, we obtain that
\begin{align*}
\sum_{j\in N(i)}e^{S_i-S_j}m_{ji}\sqrt{\rho_j}=\sum_{j\in N(i)}{m_{ij}}\sqrt{\rho_i}.
\end{align*}
This immediately implies that 
\begin{align*}
\sum_{j\in N(i)}e^{-S_j}m_{ji}\sqrt{\rho_j}=\sum_{j\in N(i)}e^{-S_i}{m_{ij}}\sqrt{\rho_i}.
\end{align*}
Now by taking $e^{S_j^*}\sqrt{\rho_j^*}=e^{S_i^*}\sqrt{\rho_i^*}$ for all $ij\in E,$ the first equation is reduced to 
\begin{align*}
\sum_{j\in N(i)}m_{ji}\rho_j^*=\sum_{j\in N(i)}m_{ij}\rho_i^*.
\end{align*}
 This leads to 
\begin{equation*}
\sum_{j\in N(i)} m_{ji}\rho^*_j + m_{ii}\rho^*_i = 0,
\end{equation*}
which is the sufficient and necessary condition that the reference process admits the stationary measure $\rho^*.$ From the above arguments, there always exists a stationary point $(\rho^*, S^*)$ which refers to the reference process itself and $\rho^0=\rho^1=\rho^*$ in the SBP.
\end{proof}

\begin{cor}\label{cor-sta}
Assume that the reference process is mass-preserving and that
there exists a  solution process of the SBP which is Markov.
Then for all $ij\in E$, $c_{ij}=\frac {e^{S_i}\sqrt{\rho_i}}{e^{S_j}\sqrt{\rho_j}}$ is the solution of 
\begin{align*}
-\sum_{k\in N(i)}c_{ki}m_{ik}+\sum_{l\in N(j)}c_{lj}m_{jl}-m_{jj}+m_{ii}=0.
\end{align*}
Meanwhile, $\rho$ is the invariant measure of the solution process in SBP.
\end{cor}

\begin{proof}
Consider the stochastic process in SBP is time homogenous Markov. We can verify that $\frac {e^{S_i}\sqrt{\rho_i}}{e^{S_j}\sqrt{\rho_j}}=c_{ij}>0$ is independent of time and that 
\begin{align*}
d_t\rho=\rho Q,
\end{align*}
where $Q_{ii}=-\sum_{j\in N(i)}{c_{ji}}m_{ij},$ $Q_{ij}=c_{ji} m_{ij}.$
Let $e^{\psi_i}=e^{S_i}\sqrt{\rho_i}$. Then it holds that 
\begin{align*}
 & \frac{\partial }{\partial t} \rho(i,t) = \sum_{j\in N(i)} -e^{\psi(j,t)-\psi(i,t)}m_{ij}\rho(i,t) +e^{\psi(i,t)-\psi(j,t)}m_{ji}\rho(j,t)\\
 & \frac{\partial }{\partial t} \psi(i,t) = -\sum_{j\in N(i)}(e^{\psi(j,t)-\psi(i,t)}-1)m_{ij}.
\end{align*}
As a consequence, for $ij\in E,$ 
\begin{align}\label{con-cij}
 d_t c_{ij}&=d_t [(e^{\psi_i-\psi_j})]\\\nonumber
 &=c_{ij}(-\sum_{l\in N(i)}e^{\psi_l-\psi_i}m_{il}+\sum_{k\in N(j)}e^{\psi_k-\psi_j}m_{jk}) 
 +c_{ij}(-m_{jj}+m_{ii})\\\nonumber
 &=c_{ij}(-\sum_{l\in N(i)}c_{li}m_{il}+\sum_{k\in N(j)}c_{kj}m_{jk}-m_{jj}+m_{ii})=0. 
\end{align}
We complete the proof for the first desired result due to $c_{ij}>0$ for $ij\in E.$

Notice that $e^{S_i-S_j}=\frac {\sqrt{\rho_j}}{\sqrt{\rho_i}}c_{ij}$ leads to 
\begin{align*}
d(S_i-S_j)=\frac 12 \frac {d\rho_j}{\rho_j}-\frac 12 \frac{d\rho_i}{\rho_i}+d \ln(c_{ij})=\frac 12 \frac {d\rho_j}{\rho_j}-\frac 12 \frac{d\rho_i}{\rho_i}.
\end{align*}
This implies that 
\begin{align*}
&-m_{ii}-\frac 12\sum_{k\in N(i)}e^{S_k-S_i}m_{ik}\frac {\sqrt{\rho_k}}{\sqrt{\rho_i}}-\frac 12\sum_{k\in N(i)}e^{S_i-S_k}m_{ki}\frac {\sqrt{\rho_k}}{\sqrt{\rho_i}}\\
&+m_{jj}+\frac 12 \sum_{l\in N(j)}e^{S_l-S_j}m_{jl}\frac {\sqrt{\rho_l}}{\sqrt{\rho_j}}+\frac 12 \sum_{l\in N(j)}e^{S_j-S_l}m_{lj}\frac {\sqrt{\rho_l}}{\sqrt{\rho_j}}\\
&=  -\frac 1{2\rho_i}(\sum_{k\in N(i)}e^{S_i-S_k}m_{ki}\sqrt{\rho_i\rho_k}
-\sum_{k\in N(i)}e^{S_k-S_i}m_{ik}\sqrt{\rho_i\rho_k}
)\\
&+ \frac 1{2\rho_j}(\sum_{l\in N(j)}e^{S_j-S_l}m_{lj}\sqrt{\rho_j\rho_l}
-\sum_{l\in N(j)}e^{S_l-S_j}m_{jl}\sqrt{\rho_j\rho_l}
),
\end{align*}
which is equivalent to
\begin{align*}
-\sum_{k\in N(i)}c_{ki}m_{ik}+\sum_{l\in N(j)}c_{lj}m_{jl}-m_{jj}+m_{ii}=0.
\end{align*}
Using \eqref{con-cij}, it yields that 
\begin{align*}
\sum_{k\in N(i)} c_{ik}m_{ki}\rho(k,t) - c_{ki}m_{ik}\rho(i,t)
= \frac 1{\rho_j}\Big(\sum_{l\in N(j)} c_{jl}m_{lj}\rho(l,t) - c_{lj}m_{jl}\rho(j,t)\Big),
\end{align*}
that is, 
\begin{align*}
d_t\rho_i=d_t \ln(\rho_j).
\end{align*}
Similarly, we have $d_t\rho_j=d_t \ln(\rho_i),$ which implies that 
\begin{align*}
d_t\rho_i=\rho_i\rho_j d_t\rho_i. 
\end{align*}
If there exists $d_t\rho_i>0,$ then we have that $\rho_i\rho_j=1.$ However, this contradicts the mass conservation $\sum_{i}\rho_i=1.$ 
It follows that $d_t\rho_i=0$, and therefore $\rho$ should be invariant with respect to time. 
We conclude that $\rho$ must be the invariant measure of the solution process in the SBP.

\end{proof}

\subsection{Discrete SBP based on minimum action with Fisher information}

{ Another way (B) to describe the discrete SBP (see e.g. \cite{chow20discrete}) lies on the discretization of the variational problem \eqref{schro0}.} {Consider the following control problem by directly discretizing the Fisher information $I(\rho)$ in \eqref{schro0}: }
 \begin{align}
 &J_1=\min_{\rho,v} \Big\{\int_{0}^1 (\frac 12\<v,v\>_{\rho}+\frac 18 I(\rho))dt +\frac 18\sum_{i} (\rho^1(i)\log(\rho^1(i))-\rho^{0}\log(\rho^0(i)))\Big\}, \label{SBP_Fisher}    \\
 & \textrm{where $\rho_i \in H^1((0,1))$, $v_{ij}\in L^2((0,1);\theta_{ij}(\rho))$ and} \nonumber \\
 & d_t\rho_t=\rho_t Q_t=-div_G(\rho_t v_t) \nonumber
 \end{align}
 with $\rho^0,\rho^1\in \mathcal P_o(G).$ In this case, we look for a stochastic process which obeys the above master equation and minimize the action with the Fisher information $I(\rho):=\frac 12\sum_{ij\in E} (\log(\rho_i)-\log(\rho_j))^2\widetilde \theta_{ij}(\rho)$. {Here $\widetilde \theta$ is another density dependent weight which may be different from earlier defined $\theta$ on the graph G}. 
 
  By using Lagrangian multiplier method, the critical point of the discrete variational approach should satisfies
\begin{align}
&v_{ij}(t)=(S_i(t)-S_j(t)),\nonumber\\
&d_t\rho_i-\sum_{j\in N(i)}(S_i-S_j)\theta_{ij}(\rho)=0, \label{ham-schrB-1} \\
&d_tS_i+\frac 12\sum_{j\in N(i)}(S_i-S_j)^2\frac{\partial \theta_{ij}}{\partial \rho_i}=\frac 18\frac{\partial}{\partial \rho_i} I(\rho).\nonumber
\end{align}
It forms a Hamiltonian system on the density space with the Hamiltonian $\frac 14\sum_{i,j}(S_i-S_j)^2\theta_{ij}(\rho)-\frac 18 I(\rho).$ In other words, the critical point gives a Hamiltonian process on the graph.

We can also reformulate the above system \eqref{ham-schrB-1} in the form of Schr\"{o}dinger system \eqref{heat}.  By taking differential on $f=\sqrt{\rho}e^{ S}$ and $g=\sqrt{\rho}e^{-S},$
we get 
\begin{align*}
&d_tf=e^{(\frac 12 \log(\rho)+S)}(\frac 12 \frac {d_t\rho}{\rho}+{d_tS})\\
&=e^{(\frac 12 \log(\rho)+S)}\Big(\frac 12 \frac {\sum_{j\in N(i)}w_{ij}(S_i-S_j)\theta_{ij}(\rho)}{\rho}-\frac 12\sum_{j\in N(i)}w_{ij}(S_i-S_j)^2\frac{\partial \theta_{ij}}{\partial \rho_i}+\frac 18 \frac{\partial}{\partial \rho_i} I(\rho) \Big),\\
&d_tg=e^{(\frac 12 \log(\rho)-S)}(\frac 12 \frac {d_t\rho}{\rho}-{d_tS})\\
&=e^{(\frac 12 \log(\rho)- S)}\Big(\frac 12 \frac {\sum_{j\in N(i)}w_{ij}(S_i-S_j)\theta_{ij}(\rho)}{\rho}+\frac 12\sum_{j\in N(i)}w_{ij}(S_i-S_j)^2\frac{\partial \theta_{ij}}{\partial \rho_i}-\frac 18 \frac{\partial}{\partial \rho_i} I(\rho) \Big).
\end{align*}
Rewriting the above systems into compact form leads to
\begin{align}
d_tf=-\frac 12 \Delta_{G} f, \label{ham-schrB-2}\\
d_tg=\frac 12 \Delta_{G}g, \nonumber
\end{align}
where $\Delta_{G}$ is the nonlinear discretization of the Laplacian operator,
 \begin{align*}
 &(\Delta_G f)_j\\
 &=-f_j \Bigg(\frac 1{f_jg_j}\sum_{l\in N(j)}\Big(\widetilde w_{jl}(\log(f_j/g_j)-\log(f_l/g_l))\widetilde \theta_{ij}(fg)+w_{jl}(\log(f_jg_j)-\log(f_l g_l)) \theta_{ij}(fg)\Big)\\
 &\qquad+\sum_{l\in N(j)}\Big(\widetilde w_{jl}|\log(f_j/g_j)-\log(f_l/g_l)|^2\frac {\partial \widetilde  \theta_{ij}(fg)}{\partial f_jg_j}+ w_{jl}|\log(f_jg_j)-\log(f_lg_l)|^2\frac {\partial \theta_{ij}(fg)}{\partial f_jg_j} \Big)
 \Bigg).
 \end{align*}
 By choosing $\theta=\theta^U,$ the above  nonlinear discretization of the heat equation gives a non-Markov random walk as a reference measure $R$ which is related to a non-Markov process in minimum problem of relative entropy $H(P|R)$.

\begin{rk}  In approach (A),  the Hamiltonian systems (\eqref{ham-schr1}, \eqref{ham-schr} and \eqref{ham-schr3}) are corresponding to the control problem \eqref{SBP_KL}, which is derived from discretizing the relative entropy $H(P|Q)$ in \eqref{schrod_original}; In approach (B), the Hamiltonian systems (\eqref{ham-schrB-1} and \eqref{ham-schrB-2}) are corresponding to the control problem \eqref{SBP_Fisher}, which is derived via discretizing the Fisher information $I(\rho)$ in \eqref{schro1}. It worth mentioning that under continuous cases, \eqref{schrod_original} and \eqref{schro1} are equivalent under the transform \eqref{Hopf-cole transform} and their corresponding Hamiltonian systems are also equivalent. However, this is not true for discrete cases. Discretizing the SBP at different stages leads to different Hamiltonian systems.
Furthermore, by comparing the systems \eqref{ham-schr3} and \eqref{ham-schrB-2} related to $f,g$, is not hard to see that the approach (A) can be viewed as a linear approximation of Laplacian operator of systems related to $f,g$ in the continuous SBP when the reference process is Markov. 
The approach (B) is related to a nonlinear approximation of of Laplacian operator in the continuous SBP. When we understand (B) in the relative entropy, it must corresponds to a nonlinear Markov process.   
\end{rk}

\subsection{Periodic marginal distribution of Hamiltonian process in SBP}

The periodic solution, as one classical topic of Hamiltonian systems, has been studied for many decades (see e.g. \cite{Che28,Rab78,MW89}).
For our considered Hamiltonian process, the periodicity of the solution appears in the density evolution. Below, we present several examples of periodic reference process, and prove that if the periodic Hamiltonian process exists, it coincides with the reference process in SBP.

By using the Floquet theorem in \cite{Tes12}, the fundamental matrix $X(t)$ satisfies  $X(t+T)=X(t)\exp(LT),$ where $\exp(LT)$ is a non-singular constant  matrix.   The Floquet exponents of $d_t \rho=\rho Q_t $ are the eigenvalues $\mu_i, i\le k\le N$ of the matrix $L$. If there exists some $i$ such that $\exp(\mu_i T)=1$ or $-1,$ then there exists periodic density function with period $T$ or $2T.$ As a consequence, we obtain the following results.

 \begin{lm}
 Assume that $\{Q_t\}_{t\ge 0}$ is transition rate matrix and $Q_t$ is T-periodic. If there exists a Floquet exponent $\mu=  \frac{k\pi\bf i}{T}, k\in \mathbb Z $,  then $d_t \rho=\rho Q_t $ has a periodic density.
 \end{lm}

 \begin{ex}
 Consider a 2-nodes graph $G$. Given a reference measure which possesses the marginal distribution as follows,
 \begin{align*}
 d_t \rho_1&=\rho_1m_{11}+\rho_2m_{21},\\
 d_t \rho_2&=\rho_1m_{12}+\rho_2m_{22},
 \end{align*}
 where $m_{21}=-m_{11},$ $m_{22}=-m_{12},$
 $m_{11}=-\frac{\frac 12-\frac 14\cos(t)+\frac 18 \sin(t)-\frac 1{16}\sin(t)\cos(t)}{(\frac 12+\frac 14\cos(t))^2}$ and $m_{22}=-\frac 1{\frac 12+\frac 14 \cos(t)}.$
 
 There exists a nontrivial periodic solution $\rho_1(t)=\frac 12+\frac 14 \cos(t), \rho_2(t)=1-\rho_1(t)$.  And the periods of $\rho_1$ and $\rho_2$ are both $T=2\pi.$
Therefore, there exists a time inhomogenous Markov process $X_t$ with periodic marginal distribution $\rho_t$ on $G$ with the transition rate matrix $Q_t=(m_{ij})_{i,j\le 2}$.
\end{ex}

We can also show the existence of time inhomogeneous Markov process with periodic marginal distribution on any fully-connected graph.  

 \begin{prop}\label{prop-exi-per}
Suppose $G$ is a fully connected graph, and $\{\rho_t\}$ is a periodic density trajectory (with period $T$) in $\mathcal{P}_{o}(G)$, then we can always find a transition rate matrix $Q(t)$ such that $\rho_t$ is the solution to the master equation $\dot{\rho_t} = \rho_tQ(t)$. 
\end{prop}
\textbf{Proof:}
Assume $G$ contains $n$ vertices.  Let us assume the non-diagonal entries of $Q(t)$ to be $\{m_{ij}\}$,  we rearrange these entries to form a $n(n-1)$ dimensional vector as:
\begin{equation*}
  m = (m_{12},...,m_{1n}, m_{21}, m_{23},...,m_{2n},...,m_{n1}, ..., m_{n\,n-1})^T.
\end{equation*}
Plugging $m$ into the Master's equation, we derive the linear equation for $m$:
\begin{equation}
    P(t) ~ m = \left(\begin{array}{cccc} \dot{\rho_1} &  \dot{\rho_2} &  \dots  &  \dot{\rho_n}  \end{array}\right)^T.   \label{Linear eq for m_ij fully connect}
\end{equation}
Where $P$ is an $n\times n(n-1)$ matrix defined as
\begin{equation*}
P(t) = \left(\begin{array}{c|c|c|c}
    P_1(t) & P_2(t) &\dots  & P_n(t)
\end{array}\right).
\end{equation*} 
\begin{equation*}
P_m(t) = \left(\begin{array}{cc}
    \rho_m(t) I_m & 0_{m\times (n-m-1)} \\
    -\rho_m(t) \boldsymbol{e}_m^T & -\rho_m(t) \boldsymbol{e}_{n-m-1}^T \\
      0_{  (n-m-1)\times m } & \rho_m(t) I_{n-m-1} 
\end{array}\right)_{n\times(n-1)} \quad \textrm{for}~1\leq m \leq n
\end{equation*}
Here we denote $\boldsymbol{e}_m^T = \underbrace{(1,...,1)}_{\text{$m$ 1s}}$. We can verify that 
\begin{equation*}
  m^0  = 
 \left(\frac{1}{(n-1)\rho_1}\boldsymbol{e}_{n-1}^T, \frac{1}{(n-1)\rho_2}\boldsymbol{e}_{n-1}^T,...,\frac{1}{(n-1)\rho_n}\boldsymbol{e}_{n-1}^T\right)
\end{equation*}
belongs to the kernel of $P(t)$, and that $P(t)$  is a full rank matrix. There must exist a solution $m^*$ to \eqref{Linear eq for m_ij fully connect}, where its entries are expressions of $\{\rho_i,\dot{\rho_i}\}_{i\in V}$. In other words, we can directly give such a solution. To be more specific, let's consider the transport process on the loop from vertex 1 to 2, 2 to 3,... n-1 to n and n to 1. This corresponds to setting $m_{ij}$ to $0$ except $m_{12},m_{23},...,m_{n-1~n},$ and $m_{n1}$. Now the equation \eqref{Linear eq for m_ij fully connect} becomes:
\begin{equation*} 
  \left(
  \begin{array}{ccccc}
  -\rho_1 &             &             &           & \rho_n  \\
              & -\rho_2 &             &          &              \\
              &              & \ddots  &          &              \\
               &             &            & \ddots &             \\
               &              &             &          &  -\rho_n
  \end{array}\right)
  \left(\begin{array}{c}
  m_{12}\\
  m_{23}\\
  \vdots \\
  m_{n-1~n} \\
  m_{n1}
  \end{array}\right) = \left(\begin{array}{c}
  \dot{\rho_1}\\
  \dot{\rho_2}\\
  \vdots \\
  \dot{\rho}_{n-1}\\
  \dot{\rho_n}
  \end{array}\right) 
\end{equation*}
Therefore the solution is $(-\frac{\dot{\rho_1}-\dot{\rho_n}}{\rho_1}, -\frac{\dot{\rho_2}}{\rho_2},\cdots,-\frac{\dot{\rho  }_{n-1}}{\rho_{n-1}},  -\frac{\dot{\rho_n}}{\rho_n})^T.$

Then we can directly take $m(t) = Km^0(t) + m^*(t)$, since $\{\rho_t\}$ is in the interior of $\mathcal{P}(G)$, we can always find a large enough $K>0$ that guarantees the entries of $m(t)$ to be always non negative. And $m(t)$ forms the transition rate matrix $Q(t)$ whose master equation admits the periodic solution $\{\rho_t\}$.\\

\begin{ex}
Consider the periodic marginal distribution  $\rho_t$:
\begin{equation*}
\rho_t = (\frac{\cos t}{2\sqrt{6}}+\frac{\sin t}{6\sqrt{2}}+\frac{1}{3}, -\frac{\cos t}{2\sqrt{6}}+\frac{\sin t}{6\sqrt{2}}+\frac{1}{3}, -\frac{\sin t}{3\sqrt{2}}+\frac{1}{3}).
\end{equation*}
which is a circle centered at $(\frac{1}{3}, \frac{1}{3},\frac{1}{3})$ with radius $\frac{1}{2\sqrt{3}}$ on $\mathcal{P}(G)$. 
Following the idea of Proposition \ref{prop-exi-per}, one may take 
\begin{align*}
  & m_{11}(t) = -\frac{6\sqrt{2} +\sqrt{3}\sin t -3\cos t}{\sqrt{3}\cos t+4\sin t+2\sqrt{2}},\; m_{12}=-m_{11}, \; m_{13}=0,\\ 
  & m_{22}(t) = -\frac{24 - 4\sqrt{2}\cos t}{-\sqrt{6}\cos t + \sqrt{2}\sin t + 4},\;
  m_{21}=-\frac 12 m_{22}, \; m_{23}=-\frac 12 m_{22}, \\
  &  m_{33}(t) = -\frac{3\sqrt{2}}{\sqrt{2}-\sin t}\; m_{13}=0, m_{23}=-m_{33}.
\end{align*}
such that $d_t=\rho_t Q_t$ with $Q_t=(m_{ij})_{i,j\le3}.$
\end{ex}

Next we aim to use general SBP \eqref{gen-sch} to produce a Hamiltonian process with periodic marginal distribution on $G.$  
In particular, when the convex function $u=x\log(x)-x-1$,
by using the Nelson's transformation $\psi_i=\sqrt{\rho_i}e^{S_i}$, the Hamiltonian system can be also rewritten as 
\begin{align*}
dS_i
&=-m_{ii}-\frac 12 \sum_{j\in N(i)}e^{S_j-S_i}m_{ij}(t)\frac {\sqrt{\rho_j}}{\sqrt{\rho_i}}-\frac 12\sum_{j\in N(i)}e^{S_i-S_j}m_{ji}(t)\frac {\sqrt{\rho_j}}{\sqrt{\rho_i}},\\\nonumber
d\rho_i&=\sum_{j\in N(i)}e^{S_i-S_j}m_{ji}(t)\sqrt{\rho_j}\sqrt{\rho_i}-\sum_{j\in N(i)}e^{S_j-S_i}m_{ij}(t)\sqrt{\rho_i}\sqrt{\rho_j}.
\end{align*}
with the Hamiltonian 
  $  \tilde{\mathcal H}(\rho,S,t)
 = \sum_{i\in V}\sum_{j\in N(i)} e^{(S_j-S_i)} m_{ij}(t)\sqrt{\rho_i\rho_j}.$
Taking $\psi$ as a time-independent potential and choosing  $\rho^0,\rho^1$ as the initial and terminal distribution from the periodic solution, then the distribution of the solution process is exactly same as the reference process. Thus it induces a Hamiltonian system which is periodic in time.
 Therefore there exists SBP with the given $\rho^0,\rho^1$ such that the solution process is Hamiltonian and its marginal distribution is periodic in time. 
 
In the following, we assume that the 
Legendre transformation $u^*$ of $u$ in \eqref{gen-sch} is continuous differentiable  and satisfies 
\begin{align*}
&u^*(x)\ge 0,\; \text{if}\; x\le 0,\;
u^*(x)\le 0,\;\text{if}\; x\ge 0,\; \\
&\frac {\partial u^*}{\partial x}(0)=1, \; \lim_{x\to-\infty} \Big|\frac {\partial u^*}{\partial x}(x)\Big|<\infty,\;\lim_{x\to+\infty} \frac {\partial u^*}{\partial x}(x) =+\infty .
\end{align*} 
Now we are able to give the characterization of the periodic Hamiltonian process on finite graph via general SBP.

\begin{tm}
Assume that the reference process is periodic with the marginal distribution and its period $T>0$.
There always exists $\rho^0,\rho^1$ such that the critical point of the general 
SBP problem \eqref{gen-sch} is a Hamiltonian process and its marginal distribution is periodic in time. 
\end{tm}
\begin{proof}
Notice that the critical point of SBP satisfies 
\begin{align*}
 & \frac{\partial }{\partial t} \rho(i,t) = \sum_{j\in N(i)} -\frac {\partial u^{*}}{\partial x}(\psi_j-\psi_i)m_{ij}\rho(i,t) +\frac {\partial u^{*}}{\partial x}(\psi_i-\psi_j)m_{ji}\rho(j,t),\\
 & \frac{\partial }{\partial t} \psi(i,t) = -\sum_{j\in N(i)}u^{*}(\psi_j-\psi_i)m_{ij},
\end{align*}
where $\rho(0)=\rho^0, \rho(1)=\rho^1.$
Choosing $\rho^0$, $\rho^1$ as two different distribution at different time of the reference process, and taking $\psi_i=\psi_j,$ we get 
\begin{align*}
 & \frac{\partial }{\partial t} \rho(i,t) = \sum_{j\in N(i)} -m_{ij}\rho(i,t) +m_{ji}\rho(j,t),\\
 & \frac{\partial }{\partial t} \psi(i,t) = 0.
\end{align*}
This implies that the critical point forms a Hamiltonian system with Hamiltonian $H(\rho,\psi,t)=\sum_{i,j}m_{ij}(t)\rho_i.$
 Due to the fact that the marginal distribution of reference process is periodic in time, 
the critical point is exactly equal to the reference process and its marginal distribution is periodic.
\end{proof}

One may wonder whether there exists certain Hamiltonian process whose marginal distribution is periodic but is not the reference process. We first use a 2-nodes graph example to point out it is not possible to get such Hamiltonian by using SBP when $u(x)=x\log(x)-x-1$. Even worse, we show that for general finite graph, the periodic Hamiltonian process exists if and only if it equals to a reference process in general SBP.

\begin{ex}\label{2node}
Given $G$ consisting of 2 nodes.
Assume the reference process with transition rate matrix $m$ is periodic with period $T>0$ and $\{t\in [0,T] | m_{ij}(t)=0, ij\in E\}$ has Lebesgue measure zero.
Notice that  $\rho, S$ of the Hamiltonian process $X(t)$ satisfies 
\begin{align*}
 & \frac{\partial }{\partial t} \rho(1,t) = -e^{\psi(2,t)-\psi(1,t)}m_{12}\rho(1,t) +e^{\psi(1,t)-\psi(2,t)}m_{21}\rho(2,t),\\
 & \frac{\partial }{\partial t} (\psi(1,t)-\psi(2,t) )= -(e^{\psi(2,t)-\psi(1,t)}-1)m_{12}+(e^{\psi(1,t)-\psi(2,t)}-1)m_{21}.
 \end{align*}
Since $m_{12}$, $m_{21}\ge 0,$ then $\psi(1)-\psi(2)$ equals to constant if and only if $\psi(1)=\psi(2).$ Meanwhile, if $\psi_1-\psi_2>0$, then $\psi_1-\psi_2$ is increasing to $+\infty$, and $\psi_1-\psi_2$ is decreasing to $-\infty$ if $\psi_1<\psi_2$.  
Then we claim that $\rho_1$ is not periodic in time. If we assume that $\rho_1$ is periodic with period $T_1,$ then it holds true 
$\int_{kT_1}^{(k+1)T_1} -e^{\psi(2,t)-\psi(1,t)}m_{12}\rho(1,t) +e^{\psi(1,t)-\psi(2,t)}m_{21}\rho(2,t) dt=0.$
Without losing generality, let us assume that $\psi_1-\psi_2>0$. It is not hard to see that 
$e^{\psi(1,t)-\psi(2,t)}$ is increasing to $+\infty$ and $e^{\psi(1,t)-\psi(2,t)}$ is decreasing to $0$ as $t\to \infty$. The boundedness of $\rho(1,t),\rho(2,t)$ yield that there exists large enough $k$ such that
\begin{align*}
\int_{kT_1}^{(k+1)T_1} -e^{\psi(2,t)-\psi(1,t)}m_{12}\rho(1,t) +e^{\psi(1,t)-\psi(2,t)}m_{21}\rho(2,t) dt>0,
\end{align*}
which leads to a contradiction. 
Therefore, $\rho(t)$ is periodic in time if and only if $\psi_1=\psi_2$. This implies that $X(t)$ is exactly the reference process.
\end{ex}

\begin{tm}\label{per-tri}
Assume the reference process with transition rate matrix $m$ is periodic with period $T>0$ and $\{t\in [0,T] | m_{ij}(t)=0, ij\in E\}$ has Lebesgue measure zero. Then the Hamiltonian process which has periodic density distribution in general SBP problem  \eqref{gen-sch} is equal to the reference process which has the periodic density distribution.
\end{tm}
\textbf{Proof.}
Assume that there is a maximum $\psi_{i^*}\ge \psi_{i},$ $i\neq i^*$ and $\psi_{i^*}>\psi_{i_{min}}$. 
Then according to the evolution of $\psi,$ 
\begin{align*} 
\frac{\partial }{\partial t} \psi(i,t) = -\sum_{j\in N(i)}u^*(\psi(j,t)-\psi(i,t))m_{ij},
\end{align*}
then the maximum principle holds, i.e., $\psi_{i^*}(t)\ge \psi_{i}(t)\ge \psi_{i_{min}}(t).$
Notice that 
\begin{align*}
\frac{d}{dt} \rho(i,t) &= \sum_{j\in N(i)} -\frac {\partial u^{*}}{\partial x}(\psi_j-\psi_i)m_{ij}\rho(i,t) +\frac {\partial u^{*}}{\partial x}(\psi_i-\psi_j)m_{ji}\rho(j,t).
\end{align*}
The periodicity of $\rho_i$ implies that there exists $T_1>0$ for any $k\in \N^+$ such that 
{\small
\begin{align*}
\int_{kT_1}^{(k+1)T_1} \sum_{j\in N(i)} -\frac {\partial u^{*}}{\partial x}(\psi_j-\psi_i)m_{ij}\rho(i,t) +\frac {\partial u^{*}}{\partial x}(\psi_i-\psi_j)e^{\psi(i,t)-\psi(j,t)}m_{ji}\rho(j,t)dt=0.
\end{align*}
}
Due to the maximum principle, if there exists one node $l$ with a local maximum of $\psi_l$ connected with another node $k$ with a local minimum of $\psi_k$, it will lead to $\psi_l-\psi_k\to +\infty $ as $t\to \infty.$ This contradicts with the periodicity of $\rho_k$ and $\rho_l$.
If any node $l$ with a local maximum of $\psi_l$ is not connected with another node $k$ with a local minimum of $\psi_k$, we pick a road $l,j_1,\cdots,j_{w},k$ which connects $l$ and $k$. Notice that $\psi_{l} \to +\infty$, $\psi_{k}\to-\infty,$ $\psi_{j_m}\in (\psi_{k},\psi_{l}),m\le w.$ Then there must exists $j_m$ such that $m$ is the smallest number which satisfies  $\psi_{k}-\psi_{j_m}\to -\infty.$  Now consider the periodicity of $\rho_{j_m}.$ 
There exists $k'$ large enough such that 
{\small
\begin{align*}
\int_{k'T_1}^{(k'+1)T_1} \sum_{j\in N(j_m)} -\frac {\partial u^{*}}{\partial x}(\psi_j-\psi_{j_m}) m_{j_mj}\rho(j_m,t) + \frac {\partial u^{*}}{\partial x}(\psi_{j_m}-\psi_{j}) m_{jj_m}\rho(j,t)dt>0.
\end{align*}
}
This leads to a contradiction, we complete the proof.

\section{More examples}
 
In this section, we conclude the paper by presenting a few more examples of Hamiltonian processes on graph.
 \begin{ex}
 (Euler-Lagrangian equations \cite{gangbo2019geodesics}) Assume that the Lagrangian in density mainifold is 
 given by $\mathcal L(\rho_t, \dot \rho_t)=\frac 12 g_W(\dot \rho_t,\dot \rho_t)-\mathcal F(\rho_t).$  Here $g_W(\sigma_1,\sigma_2):=-\sigma_1(\Delta_\rho)^+\sigma_2$ where $\sigma_k\in T_{\rho}\mathcal P_o(G), k=1,2$ and $(\Delta_\rho)^+$ is the pseudo inverse of the weight graph Laplacian matrix  $\Delta_\rho(\cdot):=div_G(\rho \nabla_G(\cdot))$. 
 Then the critical point of $$\inf_{\rho_t}\int_0^T\mathcal L(\rho_t,\partial_t\rho_t)dt$$ with given $\rho_0$ and $\rho_T$ satisfies the Euler-Lagrangian equation
 \begin{align*}
 \partial_t \frac {\delta}{\delta \partial_t\rho_t}\mathcal L(\rho_t,\partial_t\rho_t)
 =\frac {\delta}{\delta \rho_t}\mathcal L(\rho_t,\partial_t\rho_t)+C(t).
 \end{align*}
 By introducing the Legendre transform $S_t=(-\Delta_{\rho_t})^+\partial_t\rho_t,$ it can be rewritten as a Hamiltonian system. That is 
 \begin{align*}
 &\partial_t\rho_t+div(\rho\nabla_G S)=0,\\
  &\partial_tS_t+\frac 14\sum_{j\in N(i)}(S_i-S_j)^2(\partial_{\rho_i}\theta({\rho_i,\rho_j})+\partial_{\rho_i}\theta({\rho_j,\rho_i}))+\frac {\delta }{\delta \rho_t}\mathcal F(\rho_t)=C(t),
 \end{align*}
 with the Hamiltonian $\mathcal H(\rho,S)=\frac 14\sum_{ij}(S_i-S_j)^2\theta_{ij}w_{ij}+\mathcal F(\rho_t).$
 Therefore, if the transition rate matrix in generalized master equation is well-defined, the Euler-Lagrangian equation in density space determines a Hamiltonian process on $G.$ 
 \end{ex}
 
  \begin{ex}
 (Madelung system \cite{CLZ20})
 The energy is given by $$\mathcal H(\rho,S)=\frac 14\sum_{ij\in E}(S_i-S_j)^2\theta_{ij}w_{ij}+\mathcal F(\rho_t)+\beta I(\rho_t), \beta>0.$$
 Here $\mathcal F(\rho)=\sum_{i}\rho_i\mathbb V_i+\sum_{i,j}\rho_i\rho_j\mathbb W_{ij},$ and $I(\rho)=\frac 12\sum_{ij\in E}(\log(\rho_i)-\log(\rho_j))^2\widetilde \theta_{ij}.$ Here $\widetilde \theta_{ij}$ is another density dependent weight on the graph that can be the same or different from $\theta_{ij}$. 
 The Madelung system is 
 \begin{align*}
 &\partial_t\rho_t+div(\rho\nabla_G S)=0,\\
  &\partial_tS_t+\frac 14\sum_{j\in N(i)}(S_i-S_j)^2(\partial_{\rho_i}\theta({\rho_i,\rho_j})+\partial_{\rho_i}\theta({\rho_j,\rho_i}))+\frac {\delta }{\delta \rho_t}\mathcal F(\rho_t)+\beta \frac {\delta }{\delta \rho_t} I(\rho_t)=C(t).
 \end{align*}
When taking $\theta=\theta^U$, the Madelung system in density space determines a Hamiltonian process on $G.$
This system has a close relationship with the discrete Schr\"odinger equation \cite{CLZ19}.
 \end{ex}
 
 \begin{ex}($L^p$-Wasserstein distance)
The $L^p$ Wasserstein distance, $p\in(1,\infty),$ is related to the following  minimization problem, 
\begin{align*}
W_p^p(\rho^0,\rho^1)
=\inf_{v}\{\int_0^1 \sum_{i=1}^N\sum_{j\in N(i)}\frac 12\theta_{ij}(\rho)v_{ij}^p dt: \partial_t \rho +div_G(\rho v)=0, \rho(0)=\rho^0,\rho(1)=\rho^1\}.
\end{align*}
We refer to \cite{DNS09} for a continuous version of $L^p$-Wasserstein distance.
Its critical point is related to the Hamiltonian system in density space 
 \begin{align*}
 &\partial_t\rho_t+div_G( \rho_t|\nabla_G S|^{q-2}\nabla_G S)=0,\\
 &\partial_t(S_i)+\frac 1{2q}\sum_{j\in N(i)}|(\nabla_G S)_{ij}|^q(\partial_{1}\theta_{ij}+\partial_{2}\theta_{ji})=0,
 \end{align*}
with the Hamiltonian $$\mathcal H(\rho,S)=\frac 1{2q}\sum_{i,j}|\nabla_G S|^q\theta_{ij}, \frac 1q+\frac 1p=1, p\in (1,\infty).$$
When the equation of $\rho$ is determined by a transition rate matrix, this leads to a Hamiltonian process. 

\end{ex} 

\section{Acknowledgement}
The authors are very grateful to Prof. Shui-Nee Chow for introducing this problem and his helpful suggestions.
The research is partially support by Georgia Tech Mathematics Application Portal (GT-MAP) and by research grants NSF DMS-1620345,  DMS-1830225, and ONR N00014-18-1-2852.

\section{Appendix}

\subsection{The background of SBP}

Denote $\Omega=C([0,1],\mathbb R^d).$ Given $R\in M^+(\Omega)$ the law of the reversible Brownian motion (here we consider the Brownian motion with the volume Lebesgue measure, denoted by $Leb$, as the initial distribution).
Consider the relative entropy of any probability measure with respect to $R$,
$$H(P|R)=\int_{\Omega} \log(\frac{dP}{dR})dP.$$
The SBP can be formulated as 
\begin{equation}
  \min H(P|R), P\in \mathcal P(\Omega): P_0=\mu_0, P_1=\mu_1.\label{schrod_original}
\end{equation}

Here $P_0:=P(X_0\in \cdot)$, $P_1:=P(X_1\in \cdot)$ and $X_t(\omega):=\omega(t)$ is the canonical process with $\omega\in \Omega$.  It is proven (see e.g. \cite{Leo14}) that 
if $H(\widetilde \mu_0|Leb)<\infty$ and $H(\widetilde \mu_1|Leb)<\infty$, the SBP has a unique solution $\widehat P$ which enjoys the following decomposition
\begin{align*}
\widehat P= f_0(X_0) g_1(X_1) R \in \mathcal P(\Omega),
\end{align*} 
where $f_0,g_1$ are nonnegative measurable functions such that $$\mathbb E_R[f_0(X_0)g_1(X_1)]=1.$$ 
Introduce the function $f_t, g_t$ defined by
\begin{align*}
f_t(z)&:= \mathbb E_R[f_0(X_0)| X_t=z],\\
g_t(z)&:=\mathbb E_R[g_1(X_1)|X_t=z], \; P_t\text{-}a.e.,\; z\in \mathbb R^d,
\end{align*}
and the constraint 
\begin{align*}
\widetilde \mu_0=f_0g_0Leb,\;
\widetilde \mu_1=f_1g_1Leb.
\end{align*}
Then the SBP \eqref{Sch-bri-pro} with $\hbar=1$ is equivalent to the following  minimal action problem, i.e.,
\begin{align}\label{schro1}
&\inf\{H(P|R): P_0=\widetilde \mu_0, P_1=\widetilde \mu_1\}- H(\mu_0| Leb)\\\nonumber
&=\inf\Big\{\int_{0}^1\int_{\mathbb R^d}\frac {|v_t|^2}2 \mu_t(dx)dt: 
(\partial_t-\frac \Delta 2)\mu+\nabla\cdot(v \mu)=0,
 \\\nonumber
&\qquad P_0= \mu_0, P_1= \mu_1
\Big\} 
\end{align}
We denote $\rho_t$ the density of $\mu_t$ with respect to the Lebesgue  measure. 
In addition, with the assumption that $\mu_0,\mu_1$ have finite second moments, the critical point of the minimal action problem satisfies the following system 
\begin{align*}
&(\partial_t-\frac \Delta 2)\rho+\nabla\cdot(\nabla \phi \rho)=0,\; \rho(0)=\rho_0,\\
&(\partial_t+\frac \Delta 2)\phi+\frac 12 |\nabla \phi |^2=0, \; \phi(1)=\log(g_1)
\end{align*}
 with $v_t = \nabla\phi_t$. There is also a backward version of this PDE system, namely
\begin{align*}
&(-\partial_t-\frac \Delta 2)\rho+\nabla\cdot(\nabla \psi \rho)=0,\; \rho(1)=\rho_1,\\
&(-\partial_t+\frac \Delta 2)\psi+\frac 12 |\nabla \psi|^2=0, \; \psi(0)=\log(f_0).
\end{align*}
Here we have the relation $\nabla \psi_t+\nabla \phi_t=\nabla \log(\rho_t)$. 

Applying the transformation 
\begin{equation} 
  S_t=\phi_t-\frac 12 \log(\rho_t)  \label{Hopf-cole transform}
\end{equation}
as being done in \cite{Nelson19661079}, we arrive at the Hamiltonian system on the density space, 
\begin{align*} 
\frac {\partial }{\partial t}\rho+\nabla \cdot(\rho(t,x)\nabla S)&=0,\\\nonumber 
\frac {\partial }{\partial t}S+\frac 12|\nabla S|^2- \frac {1}8\frac {\delta }{\delta \rho_t}I(\rho_t)&=0.
\end{align*}
The corresponding Hamiltonian is $\mathcal H(\rho,S)=\frac 12 \int_{\mathbb R^d}|\nabla S|^2\rho dx -\frac {1}8 I(\rho)$ where $I(\rho)=\int_{\mathbb R^d}|\nabla \log(\rho)|^2\rho dx$ is the Fisher information.
Meanwhile, the action minimizing  problem \eqref{schro1} can be rewritten as 
\begin{align}\label{schro0}
&\inf_{v_t}\Big\{\int_0^1\mathbb E[\frac 12 v(t,X(t))^2]+\frac {1} 8 I (\rho(t))dt +\frac {1}2 \int (\rho^1\log(\rho^1)-\rho^0\log(\rho^0))dx \\\nonumber
&\quad | \; dX_t=v(t,X_t)dt, \;  X(0)\sim \rho^0,\; X(1)\sim \rho^1 \Big\}.
\end{align}
Here $\rho(t)$ is the density of the marginal distribution of $X_t$.

Next, by introducing the conjugate Madelung transformation $f=\sqrt{\rho}e^{S}, g=\sqrt{\rho}e^{-S}$ ( also known as ``Hopf-Cole'' transformation), $f$ and $g$ satisfy so-called ``Schr\"{o}dinger system" (see e.g. \cite{CLZ20, conforti2017extremal, Blaquire1992ControllabilityOA}),
\begin{align}
(\partial_t-\frac \Delta 2) g=0, \; g(0)=g_0,\label{heat} \\
(\partial_t+\frac \Delta 2) f=0, \; f(1)=f_1.\nonumber
\end{align}
This also implies the following relationships
\begin{align*}
\phi=\log(f)=S+\frac 12\log(\rho),
\psi=\log(g)=-S +\frac 12\log(\rho).
\end{align*}

\subsection{Discrete Dynamical formula of SBP}
Based on the relative entropy from the SBP,
\begin{equation*}
  \min_{P} H(P|R)
\end{equation*}
with $P$  the path measure of the random process with transition rate matrix $\{\widehat{m}^t_{ij}\}$ and $R$ the reference measure with transition matrix $\{{m}^t_{ij}\}$. We prove that
\begin{equation*}
  H(P|R)= H(\rho(\cdot,0) | \widetilde {\rho}(\cdot,0))+\int_0^1 \sum_{i\in V} \rho(i,t)  \sum_{j\in N(i)} \left(\frac{\widehat {m}^t_{ij}}{m_{ij}^t}\log\left(\frac{\widehat {m}^t_{ij}}{m_{ij}^t}\right) -\frac{\widehat {m}^t_{ij}}{m_{ij}^t} + 1\right)m_{ij}^t.  
\end{equation*} 
This result is introduced in \cite{girsanov1960transforming}. It can be treated as the discrete version of Girsanov Theorem on finite graph. Since we are not able to find a direct proof of this result in the related references, we provide a proof to this result as follows:

\begin{proof}
Let us consider the reference Markov process with path measure $R$ on graph $G$ with continuity equation
\begin{equation*}
  \frac{\partial \widetilde {\rho}(i, t)}{\partial t} = \sum_{j\in N(i)} {m}_{ji}\widetilde {\rho}(j,t) - {m}_{ij}\widetilde {\rho}(i,t), \quad   i\in V.
\end{equation*}
We aim to find a Markov process with path measure $P$ on $G$ with continuity equation
\begin{equation*}
  \frac{\partial\rho(i,t)}{\partial t} = \sum_{j\in N(i)}\widehat m_{ji}\rho(j,t)-\widehat m_{ij}\rho(i,t), \quad i\in V,
\end{equation*}
as a solution process of SBP.
We approximate the relative entropy between $P$ and $R$:
\begin{equation*}
   H(P|R) = \int \log\left(\frac{P}{R}\right)P ~ d\gamma
\end{equation*}
by discretizing the time interval $[0,1]$ into $N$ equal small intervals with length equals to $h=\frac{1}{N}$. 
Then we also use the a time discrete Markov process to approximate a time-continuous Markov process. 
For reference process $R$, the transition matrix is approximated by 
\begin{equation*}
 {\pi}_{ij}=\begin{cases}
  {m}_{ij}h \quad \textrm{if}~ j\neq i,   \\
       1-h\left(\sum_{l\in N(i)} {m}_{ij}\right) \quad \textrm{if} ~ j=i.
  \end{cases}
\end{equation*}
Similarly we can define an approximation of transition rate matrix denoted by $\{\widehat \pi_{ij}\}$ for $P$.

Now for a specific (time-discrete) path $\gamma = \{v_0,...,v_N\}$ on the graph, let us compute $P(\gamma)$ and $R(\gamma)$, where 
\begin{equation*}
  R(\gamma) = \widetilde {\rho}(v_0,0){\pi}_{v_0,v_1}...\pi_{v_{N-1}v_N},   \quad P(\gamma) = \rho(v_0,0)\widehat \pi_{v_0,v_1}...\widehat \pi_{v_{N-1}v_N}
\end{equation*}
Then we have
\begin{equation*}
  \log\left(\frac{P(\gamma)}{R(\gamma)}\right) = \log\left(\frac{\rho(v_0,0)}{\widetilde {\rho}(v_0,0)}\right)+\sum_{k=0}^{N-1}\log\left(\frac{\widehat \pi_{v_{k}v_{k+1}}}{{\pi}_{v_kv_{k+1}}}\right).
\end{equation*}
Therefore, the relative entropy is
\begin{align*}
    & H(P|R)  = \sum_{\gamma, \textrm{path on G with length} ~ N+1}  \log\left(\frac{P(\gamma)}{R(\gamma)}\right)P(\gamma) \\
  = & \sum_{\{v_0,...,v_N\}} \left(\log\left(\frac{\rho(v_0,0)}{\widetilde {\rho}(v_0,0)}\right)+\sum_{k=0}^{N-1}\log\left(\frac{\widehat \pi_{v_{k}v_{k+1}}}{{\pi}_{v_kv_{k+1}}}\right)\right) \rho(v_0,0)\widehat \pi_{v_0,v_1}...\widehat \pi_{v_{N-1}v_N} \\
  = & \sum_{v_0} \log\left(\frac{\rho(v_0,0)}{\widetilde {\rho}(v_0,0)}\right) \rho(v_0,0) \underbrace{\left(\sum_{\{v_1,...,v_N\} }\widehat \pi_{v_0v_1}...\widehat \pi_{v_{N-1}v_N}\right)}_{\text{=1}} \\
 + & \sum_{k=0}^{N-1} \underline{\sum_{v_k\in V}  ~  \sum_{v_{k+1}\in N(v_k)\bigcup\{v_k\}}\log\left(\frac{\widehat \pi_{v_kv_{k+1}}}{{\pi}_{v_kv_{k+1}}}\right)\underbrace{\left(\sum_{\{v_0,...,v_k\}} \rho(v_0,0)\widehat \pi_{v_0v_1}...\widehat \pi_{v_{k-1}v_k} \right)}_{\text{=$\rho(v_k,t_k)$}}}\\
 &\underline{ \times \widehat \pi_{v_kv_{k+1}} \underbrace{\left(\sum_{\{v_{k+1},...,v_N\}}\widehat \pi_{v_{k+1}v_{k+2}}...\widehat \pi_{v_{N-1}v_N}\right).}_{\text{=$1$}}} 
\end{align*}

Now we take out the underlined part, it can be simplified as following:
\begin{align*}
  & ~ \sum_{v_k\in V}~  \sum_{v_{k+1}\in N(v_k)\bigcup\{v_k\}} \log\left(\frac{\widehat \pi_{v_kv_{k+1}}}{{\pi}_{v_kv_{k+1}}}\right) \rho(v_k, t_k)\widehat \pi_{v_kv_{k+1}} \\
  & = \sum_{v_k\in V} \rho(v_k,t_k)\left(\sum_{v_{k+1}\in N(v_k)}\log\left(\frac{\widehat \pi_{v_kv_{k+1}}}{{\pi}_{v_kv_{k+1}}}\right)\widehat \pi_{v_kv_{k+1}} + \log \left(\frac{\widehat \pi_{v_kv_k}}{{\pi}_{v_kv_k}}\right)\widehat \pi_{v_kv_k} \right). 
\end{align*}
Notice that 
\begin{equation}
  \frac{\widehat \pi_{v_kv_{k+1}}}{{\pi}_{v_kv_{k+1}}} = \frac{h \widehat m_{v_kv_{k+1}}}{h {m}_{v_kv_{k+1}}} = \frac{\widehat m_{v_kv_{k+1}}}{{m}_{v_kv_{k+1}}}, \label{eq1}
\end{equation}
and 
\begin{equation*}
  \frac{\widehat \pi_{v_kv_k}}{{\pi}_{v_kv_k}} = \frac{1-h\left(\sum_{u\in N(v_k)} \widehat m_{v_k u}\right)}{1-h\left(\sum_{u\in N(v_k)} {m}_{v_k u}\right)}=1+h\left(\sum_{u\in N(v_k)}{m}_{v_k u} - \sum_{u\in N(v_k)} \widehat m_{v_k u}\right) + O(h^2).
\end{equation*}
As result, 
\begin{equation}
  \log \frac{\widehat \pi_{v_kv_k}}{{\pi}_{v_kv_k}} = h\left(\sum_{u\in N(v_k)}{m}_{v_k u} - \sum_{u\in N(v_k)} \widehat m_{v_k u}\right) + O(h^2). \label{eq2}
\end{equation}
Now plugging \eqref{eq1} and \eqref{eq2} into our previous computations, we get:
\begin{align} \nonumber 
  & \sum_{v_k\in V} \rho(v_k,t_k)\Bigg(\sum_{v_{k+1}\in N(v_k)}\log\left(\frac{\widehat m_{v_kv_{k+1}}}{m_{v_kv_{k+1}}}\right)\cdot \underbrace{h \widehat m_{v_kv_{k+1}}}_{\text{$\pi_{v_kv_{k+1}}$}} \\\nonumber 
  &\quad + \Bigg(h\left(\sum_{u\in N(v_k)}{m}_{v_k u} - \sum_{u\in N(v_k)} \widehat m_{v_k u}\right)+ O(h^2)\Bigg)\cdot\underbrace{(1-O(h))}_{\text{$\pi_{v_kv_k}$}} \Bigg)  \nonumber \\
 = &  h\left( \sum_{v_k\in V} \rho(v_k,t_k)\left( \sum_{u\in N(v_k)} \log\left(\frac{\widehat m_{v_ku}}{{m}_{v_ku}} \right) \widehat m_{v_ku}+{m}_{v_k u}  - \widehat m_{v_ku} \right) \right) + O(h^2) \nonumber \\ \nonumber 
 =& h\left( \sum_{v_k\in V} \rho(v_k,t_k)  \sum_{u\in N(v_k)} \left(\log\left(\frac{\widehat m_{v_ku}}{{m}_{v_ku}}\right) + \frac{ {m}_{v_k u}}{\widehat m_{v_ku}}  - 1\right) \widehat m_{v_ku}  \right) + O(h^2). \label{equ3}
\end{align}
It follows that
\begin{align*}
  &H(P|R)\\
  &= \sum_{v_0} \log\left(\frac{\rho(v_0,0)}{\widetilde {\rho}(v_0,0)}\right) \rho(v_0,0) + \sum_{k=0}^{N-1}  h\left( \sum_{v_k\in V} \rho(v_k,t_k)  \sum_{u\in N(v_k)} \left(\log\left(\frac{J_{v_ku}}{\widehat{J}_{v_ku}}\right) + \frac{\widehat{J}_{v_k u}}{J_{v_ku}}  - 1\right) J_{v_ku}  \right) + O(h)  \\
  &=H(\rho(\cdot,0) |\widehat{\rho}(\cdot,0))+\sum_{k=0}^{N-1}  h\left( \sum_{i\in V} \rho(i,t_k)  \sum_{j\in N(i)} \left(\log\left(\frac{\widehat m_{ij}}{{m}_{ij}}\right) + \frac{{m}_{ij}}{\widehat m_{ij}}  - 1\right) \widehat m_{ij}  \right) + O(h).  
\end{align*}
Finally, let us sent $N\rightarrow +\infty$ and thus $h\rightarrow 0$, the relative entropy between $P$ and $R$ equals to:
\begin{equation*}
  H(P|R) = H(\rho(\cdot,0)|\widehat {\rho}(\cdot,0)) + \int_{0}^{1}\sum_{i\in V} \rho(i,t_k)  \sum_{j\in N(i)} \left(\log\left(\frac{\widehat m_{ij}}{ {m}_{ij}}\right) + \frac{{m}_{ij}}{\widehat m_{ij}}  - 1\right) \widehat m_{ij} ~dt. 
\end{equation*}

\end{proof}

\bibliographystyle{plain}
\bibliography{bib}
\end{document}